\documentclass[12pt]{amsart}
\usepackage{graphicx}
\usepackage{color}
\usepackage{latexsym}
\usepackage{amscd}
\usepackage{epsfig}
\usepackage{cancel}
\usepackage{todonotes}
\usepackage{float}
\usepackage{amsmath, amssymb, amsthm, amsfonts, amsgen}
\usepackage{stmaryrd}
\usepackage{epstopdf}
\usepackage{caption}
\usepackage{subcaption}
\usetikzlibrary{decorations.pathreplacing}
\usetikzlibrary{decorations.pathreplacing,shapes,shadows,arrows,positioning,graphs,calc,patterns}
\usepackage{comment}
\usepackage{url}
\usepackage{hyperref}

\graphicspath{ {./} }

\numberwithin{equation}{section}
\newtheorem{definition}{Definition}[section]
\newtheorem{lemma}[definition]{Lemma}
\newtheorem{theorem}[definition]{Theorem}
\newtheorem{proposition}[definition]{Proposition}
\newtheorem{corollary}[definition]{Corollary}
\newtheorem{remark}[definition]{Remark}

\newcommand{\md}{{\rm d}}
\newcommand{\dx}{\, \md x}
\newcommand{\dy}{\, \md y}
\newcommand{\drho}{\, \md \rho} 
\newcommand{\dphi}{\, \md \phi} 
\newcommand{\dtheta}{\, \md \theta} 
\newcommand{\dz}{\, \md z}

\newcommand{\R}{\mathbb{R}}
\newcommand{\N}{\mathbb{N}}
\newcommand{\C}{\mathbb{C}}

\newcommand{\om}{\Omega}

\newcommand{\cof}{{\rm cof}\,}

\newcommand{\1}{{\bf 1}}

\newcommand{\eps}{\epsilon}

\newcommand{\sca}{\mathcal{A}}

\newcommand{\spt}{\textrm{spt}\,}

\newcommand{\D}{\mathbb{D}}

\newcommand*{\lc}{\mbox{\large$\chi$}}

\newcommand{\re}{\mathrm{Re\,}}
\newcommand{\wto}{\rightharpoonup}

\def\Xint#1{\mathchoice
   {\XXint\displaystyle\textstyle{#1}}%
   {\XXint\textstyle\scriptstyle{#1}}%
   {\XXint\scriptstyle\scriptscriptstyle{#1}}%
   {\XXint\scriptscriptstyle\scriptscriptstyle{#1}}%
   \!\int}
\def\XXint#1#2#3{{\setbox0=\hbox{$#1{#2#3}{\int}$}
     \vcenter{\hbox{$#2#3$}}\kern-.5\wd0}}

\def\dashint{\Xint-}

\newcommand\EEE{\color{black}}
\newcommand{\mk}{\color{black}}

\newcommand{\referee}{\color{black}}

 \def\FS#1#2#3#4 {\scalebox{0.7}{{\begin{tabular}{|c|c|} \hline $#2$ &$#1$ \\ \hline $#3$ & $#4$ \\ \hline \end{tabular}}}}
 \def\INS#1#2#3#4 {\scalebox{0.8}{{\begin{tabular}{|c|c|c|c|} \hline $#1$ &$#2 $ & $#3$ & $#4$   \\ \hline \end{tabular}}}}

\begin{document}

\title[Hadamard's inequality in the mean]
{Hadamard's inequality in the mean}

\author[J. Bevan]
{Jonathan Bevan}
\address[J. Bevan]{Department of Mathematics, University of Surrey, Guildford, GU2 7XH, United Kingdom.}
\email{j.bevan@surrey.ac.uk}

\author[M. Kru\v{z}\'{i}k]
{Martin Kru\v{z}\'{i}k}
\address[M. Kru\v{z}\'{i}k]{ Czech Academy of Sciences, Institute of Information Theory and Automation,
Pod vod\'arenskou v\v{e}\v z\'\i\ 4, 182 00, Prague 8, Czechia.}
\email{kruzik@utia.cas.cz}

\author[J. Valdman] {Jan Valdman} 
\address[J. Valdman]{Czech Academy of Sciences, Institute of Information Theory and Automation, 
Pod vod\'arenskou v\v{e}\v z\'\i\ 4, 182 00, Prague 8, Czechia.
}
\email{jan.valdman@utia.cas.cz}


\maketitle

\begin{abstract}  Let $Q$ be a Lipschitz domain in $\R^n$ and let $f \in L^{\infty}(Q)$.  We investigate conditions under which the functional $$I_n(\varphi)=\int_Q |\nabla \varphi|^n+ f(x)\det \nabla \varphi \dx $$ obeys $I_n \geq 0$ for all $\varphi \in W_0^{1,n}(Q,\R^n)$, an inequality that we refer to as Hadamard-in-the-mean, or (HIM). We prove that there are piecewise constant $f$ such that (HIM) holds and is strictly stronger than the best possible inequality that can be derived using the Hadamard inequality $n^{\frac{n}{2}}|\det A|\leq |A|^n$ alone.  When $f$ takes just two values, we find that (HIM) holds if and only if the variation of $f$ in $Q$ is at most $2n^{\frac{n}{2}}$.  For more general $f$, we show that (i) it is both the geometry of the `jump sets' as well as the sizes of the `jumps' that determine whether (HIM) holds and (ii) the variation of $f$ can be made to exceed $2n^{\frac{n}{2}}$, provided $f$ is suitably chosen.  
Specifically, in the planar case $n=2$ we divide $Q$ into three regions $\{f=0\}$ and $\{f=\pm c\}$, and prove that as long as $\{f=0\}$ `insulates' $\{f= c\}$ from $\{f= -c\}$ sufficiently, there is $c>2$ such that (HIM) holds.  Perhaps surprisingly, (HIM) can hold even when the insulation region $\{f=0\}$ enables the sets $\{f=\pm c\}$ to meet in a point.  As part of our analysis, and in the spirit of the work of Mielke and Sprenger \cite{mielke-sprenger}, we give new examples of functions that are quasiconvex at the boundary.

\end{abstract}

\tableofcontents

\maketitle

\section{Introduction}

Let $Q$ be a Lipschitz domain in $\R^n$ and let $f\in L^\infty(Q)$.  This paper concerns the functional
\begin{align}\label{eye}I_n(\varphi)& = \int_Q |\nabla \varphi|^n+ f\det \nabla \varphi \dx \end{align}
defined on $W_0^{1,n}(Q;\R^n)$ for $n\ge 2$. By a Hadamard-in-the-mean inequality, henceforth (HIM), we mean inequality of the form  
\begin{align}\label{him}\tag{HIM}
I_n(\varphi) \geq 0 \quad \quad \forall \ \varphi \in W_0^{1,n}(Q,\R^n).
\end{align}
The fact that (HIM) holds for more general $f$ is already indicated by the observation that $\int_Q\det \nabla \varphi \dx=0$ whenever $\varphi \in W_0^{1,n}(Q,\R^n)$ because the map   $\R^{n\times n}\to\R$: $F\mapsto \det F$ is a so-called null Lagrangian; cf.~e.g.~\cite{dacorogna}. Hence, if $f$ is a constant function then (HIM) always holds. 

The classical pointwise Hadamard inequality, which is easily proved through the use of the arithmetic-geometric mean inequality, is \begin{align}\label{pwhad} c_n|\det A| \leq |A|^n, 
\end{align}
where $c_n:=n^\frac{n}{2}$, $A\in\R^{n\times n}$ and $|A|^2=\sum_{i,j=1}^n A^2_{ij}$. The equality holds for the identity matrix, for instance, i.e., $c_n$ is the largest number  for which \eqref{pwhad} holds.
Using this and the fact that $\int_Q\det \nabla \varphi \dx =0$,  we find that (HIM) holds for those essentially bounded and measurable functions  $f$ which obey 
\begin{align}\label{island1}||f-\bar{f}_{Q}||_\infty \leq  n^{\frac{n}{2}}, \end{align}
where $\bar{f}_{Q}=\dashint_Q f\dx$.
This condition is sharp and, moreover, it is necessary and sufficient for the sequential weak lower semicontinuity of the functional $I_n$  when $f$ is a two-state function, i.e. a function of the form $f=M\chi_{Q'}$ and $Q'\subset Q$; see 
Proposition~\ref{Prop:swlsc}. This is to be contrasted with the fact that if $f\in C(\bar Q)$ then $I_n$ is weakly lower semicontinuous on  $W^{1,n}_0(Q,\R^n)$ (see \cite[Thm.~2.9, Thm.~2.10]{kalamajska-kruzik}).  In the planar case, the necessity of \eqref{island1} relies, in particular, on the results of Mielke and Sprenger \cite{mielke-sprenger}, which characterize the quadratic forms in $\R^{2\times 2}$ that are quasiconvex at the boundary \cite{bama}. In higher dimensions, and in the specific case of the functionals that we consider, we prove in Proposition \ref{prop:n_dim_two_state} a similar result.  These constitute new examples of functions that are quasiconvex at the boundary \cite{bama} in any dimension.  

We note that a ``quartic version'' of (HIM) for $n=2$ and a constant $f$, namely 
\begin{align}\label{qhim}
 \int_Q |\nabla \varphi|^2(|\nabla \varphi|^2+ \gamma\det \nabla \varphi) \dx \geq 0 \quad \quad \forall \ \varphi \in W^{1,4}_0(Q,\R^2)\ ,
\end{align}
is satisfied for every $|\gamma|<2+\epsilon$ if $\epsilon>0$ is small enough, as is shown in \cite{alibert-dacorogna}. 
\referee In a similar vein, and provided one takes $f$ to be the constant $f=n^{\frac{n}{2}}$, in \cite[Cor. 9.1]{iwaniec-lutoborski-ARMA} it is shown that for each $a\geq 0$ and $n \in \N$, there is a constant $\delta=\delta(n,a) \in [0,1)$ such that if $\om \subset \R^n$
\begin{align}\label{mc}
    \int_{\om} |\nabla \varphi|^a (|\nabla \varphi|^n - n^{\frac{n}{2}} \det \nabla \varphi) \dx \geq (1-\delta) \int_{\om} |\nabla \varphi|^{n+a} \dx \quad \forall \varphi \in W^{1,n+a}_0(\om,\R^n).
\end{align}
The authors of \cite{iwaniec-lutoborski-ARMA} refer to this as mean coercivity, and we can link our work to it in the following way.  For example, we show in Proposition \ref{islandproblem} that for two state $f$ of the form $f=M \chi_{\om}$, where $\om \subset Q$ is measurable and satisfies a mild technical condition, the functional 
\begin{align}
    I_n(\varphi):=\label{islandpreview} \int_{Q}|\nabla \varphi|^n + M \chi_{_{\om}}\det \nabla \varphi \dx 
\end{align}
is nonnegative on $W_0^{1,n}(Q;\R^n)$ if and only if $|M| \leq 2c_n$.  It follows that if $M < 2c_n$ then in fact 
\begin{align*}
    I_n(\varphi) \geq \left(1-\frac{|c_n|}{2}\right)\int_Q|\nabla \varphi|^n\dx
\end{align*}
for all $\varphi \in W_0^{1,n}(Q;\R^n)$, which is a form of mean coercivity in which we can even specify the prefactor in the right-hand side.  One can do likewise with the functionals we consider in Section \ref{insulation}.  We also refer the reader to \cite{iwaniec, iwaniec-lutoborski-SIMA}, in particular, for their deep results on mean coercivity and the pivotal role played there by Hadamard's pointwise inequality.

\EEE  One can also view the functional in \eqref{eye} as a general form of an `excess functional' associated with an energy $E$ and a suitably-defined stationary point $u_0$, say, so that 
$$E(u) = E(u_0)+I_n(\varphi),$$
with $\varphi=u-u_0$. This is the situation discussed in  \cite{Bevan-14} and \cite{BD22} where, in both cases, the functional $I_n$ is of the form \eqref{eye}, $f=C\ln(|\cdot|)$, $C$ is constant, and the domain of integration is the unit ball in $\R^2$.  For large enough $C$, \cite[Proposition 3.5 (i)]{Bevan-14} shows that (HIM) fails;  by contrast, it can be deduced from \cite[Theorem 1.2]{BD22} that, for sufficiently small $C$, (HIM) holds.  

If $f$ is positively zero homogeneous, i.e., $f(x)=f(kx)$ for all $x\in Q$ and all $k\ge 0$ then $I_n\ge 0$ on $W^{1,n}_0(Q;\R^n)$ is a necessary condition for sequential weak lower semicontinuity of $I$ on that space. Indeed, 
we can assume that $B(0,1)\subset Q$ where $B(0,1)$ is the unit ball in $\R^n$ centered at the origin.  Then extend  $\varphi\in W^{1,n}_0(B(0,1);\R^n)$ by zero to the whole space.  Note that if we take $\varphi_k(x)=\varphi(kx)$ then get $I_n(\varphi_k)=I_n(\varphi)$ for all $k$, while $\varphi_k \rightharpoonup 0$ still holds  and  we get 
$\nabla \varphi_k\to 0$ in measure.
However, $\lim_{k\to\infty} I_n(\varphi_k)=I_n(\varphi)\ge 0$ by lower semicontinuity. Hence, \eqref{him} can be seen as an inhomogeneous version of (Morrey's) quasiconvexity \cite{morrey}.

Our ultimate goal is to find conditions on $f$ that are both necessary and sufficient for the (HIM) inequality to hold.  While we  consider this task complete in the case of two-state $f$, we are still far from such a characterization in general.  Thus, we study a variety of choices of $f$ for which (HIM) holds and is, in particular, strictly better than the best possible inequality that can be derived using the pointwise Hadamard inequality \eqref{pwhad}
alone.  We find that in these more complex cases, each corresponding to a choice of piecewise continuous $f$,
it is a subtle combination of the `geometry' of the subdomains on which $f$ is constant and the sizes of the jumps themselves that determine whether (HIM) holds. For instance, let $Q=[-1,1]^2$, let $Q_1$ be the first (or positive) quadrant, and let $Q_2,Q_3,Q_4$ be the other quadrants labeled anticlockwise from $Q_1$, see Fig.\ref{pic:intro}.  

\begin{figure}
\centering
\begin{minipage}{1\textwidth}
\centering
\begin{tikzpicture}[scale=2]
\node (A) at (-2,-2) {}; 
\node[right=2 of A.center] (B) {};
\node[right=2 of B.center] (C) {};
\node[right=2 of C.center] (D) {};
\node[above=2 of A.center] (A') {};
\node[above=2 of B.center] (B') {};
\node[above=2 of C.center] (C') {};
\node[above=2 of A'.center] (A'') {};
\node[above=2 of B'.center] (B'') {};
\node[above=2 of C'.center] (C'') {};

\filldraw[thick, top color=gray!30!,bottom color=gray!30!] (A.center) rectangle node{$+\sqrt{8}$} (B'.center);
\filldraw[thick, top color=gray!30!,bottom color=gray!30!] (B'.center) rectangle node{$-\sqrt{8}$} (C''.center);
\filldraw[thick, top color=white,bottom color=white] (A'.center) rectangle node{$0$} (B''.center);
\filldraw[thick, top color=white,bottom color=white] (B.center) rectangle node{$0$} (C'.center);

\end{tikzpicture}
\end{minipage}
\caption{Distribution of $f$ 
yielding $I(\varphi) \geq 0$ for any $\varphi \in W^{1,2}_0(Q;\R^2)$ .} \label{pic:intro}
\end{figure}
\noindent
Then the corresponding functional 
\begin{align*}
I_2(\varphi)=\int_Q |\nabla \varphi|^2 \dx  + \sqrt{8}\int_{Q_3}\det \nabla \varphi \dx - \sqrt{8} \int_{Q_1} \det \nabla \varphi \dx 
\end{align*}
is nonnegative on $W^{1,2}_0(Q;\R^2)$.  See Proposition \ref{foursquares1} for details. 
Such a result is impossible to prove, at least as far as we can tell, using the pointwise Hadamard inequality alone, and the example is optimal in the sense that if $\pm\sqrt{8}$ is replaced by $\pm(\sqrt{8}+\eps)$ for any positive $\eps$ then there are $\varphi$ such that $I_2(\varphi)<0$. 
This and other examples are discussed in more detail in Section \ref{insulation}.  

\subsection{Notation}
For clarity, we use $\lc_S$ to represent the characteristic function of a set $S$.  An $n \times n$ matrix is conformal provided $|A|^n = c_n \det A$, and anticonformal if $|A|^n = -c_n \det A$.  Further, $\mathcal{L}^n$ denotes the $n$-dimensional Lebesgue measure.  All other notation is either standard or else is defined when first used.  \referee The $2\times 2$ matrix representing a rotation $\pi/2$ radians anticlockwise is denoted by $J$.  \EEE

 \section{Preliminary remarks}

At this point, we collect together some basic features of functionals of the form 
\begin{align}\label{iQ}
    I_n(\varphi):=\int_Q |\nabla \varphi|^n+f\det \nabla \varphi \dx, \quad \varphi \in W_0^{1,n}(Q;\R^n), 
\end{align}
where $f \in L^{\infty}(Q)$ and $Q\subset\R^n$ is a bounded Lipschitz domain.  

\begin{proposition}\label{portmanteau} Let $I_n$ be given by \eqref{iQ}.  Then the following are true:
\begin{itemize}
    \item[(i)] $I_n(\lambda \varphi)=\lambda^n I_n(\varphi)$ for any $\lambda>0$, and $I_n(\varphi)$ is finite for every $\varphi \in W_0^{1,n}(Q;\R^n)$.
    \item[(ii)]  If $\varphi^*$ is a stationary point then $I_n(\varphi^*)=0$. 
    \item[(iii)]  \referee In general, the functional $I_n$ is not invariant with respect to changes of domain $Q$.  If $Q'=F(Q)$, where $F$ is conformal, then 
\begin{align}\label{Arnold}I_n(\varphi;Q):=\int_{Q'}|\nabla \Phi(y)|^n+f'(y)\det \nabla \Phi(y) \dy,\end{align}
    where $\Phi:=\varphi\circ F^{-1}$ and $f':=f \circ F^{-1}$. \EEE
    \item[(iv)] If $f$ is constant then $I_n(\varphi)\ge 0$ for every  $\varphi\in W^{1,n}_0(Q;\R^n)$.
    \end{itemize}

\end{proposition}

\begin{proof}Part (i) is clear. \\
To see (ii), suppose $\varphi^* \in W_0^{1,n}(Q;\R^n)$ is a stationary point of the functional, so that 
\begin{align}\label{naiveEL}\int_Q n|\nabla \varphi^*|^{n-2}\nabla \varphi^* \cdot \nabla \zeta +f \,  \cof \nabla \varphi^* \cdot \nabla \zeta \dx = 0, \quad \zeta \in C_c^\infty(Q;\R^n) . \end{align}
Since $C_c^\infty(Q)$ is dense in  $W_0^{1,n}(Q;\R^n)$, a straightforward approximation argument shows that the weak form \eqref{naiveEL} holds for $\zeta$ in $W_0^{1,n}(Q;\R^n)$.  Now take  $\zeta=\varphi^*$ and apply the identity $\cof \nabla \varphi^* \cdot \nabla \varphi^* = n \det \nabla \varphi^*$, so that \eqref{naiveEL} becomes $nI_n(\varphi^*)=0$. Hence (ii).  \\
For (iii), \eqref{Arnold} follows by making the prescribed change of variables and bearing in mind that $\nabla F^T \nabla F = (\det \nabla F) \1$ when $F$ is a conformal map. It should now be clear that changing the domain changes $f$.

Finally, (iv) follows from the fact that $\int_Q\det\nabla\varphi\dx=0$, cf.~\cite{dacorogna}, for instance.
\end{proof}

We remark that (i) or  (ii) imply in particular that if the functional $I_n$ attains a minimum in $W_0^{1,n}(Q;\R^n)$ then the value of the minimum is zero and (HIM) is automatic.  Note, however, that the functional is not necessarily `coercive enough' to guarantee that a minimum is attained.  This is one of the reasons why (HIM) is interesting:  we obtain results that are consistent with but do not depend upon the coercivity just mentioned.  

We now derive a necessary condition on $f$ for the functional $I_n$ to be nonnegative.  Let $O(n)$ be the set of rotations and reflections in $\R^{n\times n}$.  Let $\varphi\in W^{1,n}_0(\Omega;\R^n)$ satisfy $\nabla\varphi\in O(n)$ almost everywhere in $\Omega\subset Q$ where $\Omega\subset\R^n$ is a Lipschitz domain. The existence of such $\varphi$ follows from \cite[p.~199 and Rem.2.4]{dacorogna-marcellini}.
Let $\Omega^\pm=\{x\in \Omega: \det\nabla\varphi(x)=\pm 1\}$ and observe that, since $\det \nabla \varphi$ is a null Lagrangian, it must hold that  $\mathcal{L}^n(\Omega^-)=\mathcal{L}^n(\Omega^+)=\mathcal{L}^n(\Omega)/2$.
Hence 
\begin{align*}
I_n(\varphi)&=n^{n/2}\mathcal{L}^n(\Omega)+\int_{\Omega^+}f(x) \dx-\int_{\Omega^-}f(x) \dx
\end{align*} 
rearranges to
\begin{align}\label{4+}
I_n(\varphi)=\frac{\mathcal{L}^n(\Omega)}{2}\left(2n^{n/2}+\dashint_{\Omega^+}f(x)\dx -\dashint_{\Omega^-}f(x)\dx \right),\end{align}
which thereby forms a necessary condition for $I_n\ge 0$. Here, $\dashint$ denotes the mean value of the integrand on the respective set.  Adapting this idea leads to the following result.

\begin{proposition}
Let $I_n$ be given by \eqref{iQ} and assume that $I_n \geq 0$ on $W^{1,n}_0(Q;\R^n)$. Let $\om \subset Q$ be a Lipschitz domain and let $\varphi\in W^{1,n}_0(\Omega;\R^n)$ satisfy $\nabla\varphi\in O(n)$ almost everywhere in $\Omega$.
Then
it is necessary that  
\begin{align}\label{meanest}
    \left|\dashint_{\om^+}f(x) \dx-\dashint_{\om^-}f(x) \dx \right| & \leq 2n^{n/2}.
\end{align}
\end{proposition}

\begin{proof}
The argument above shows that 
\begin{align}\label{astor}I_n(\varphi)=\frac{\mathcal{L}^n(\om)}{2}\left(2n^{n/2}-\dashint_{\om^+} f(x) \dx +\dashint_{\om^-} f(x) \dx \right).\end{align}
Replacing $\varphi$ by $\phi:=(-\varphi_1,\varphi_2,\ldots)$, we find that the roles of $\om^+$ and $\om^-$ are exchanged, i.e.
$$\om^{\pm}=\{x \in \om: \ \det \nabla \phi = \mp 1\},$$
which leads directly to
\begin{align}\label{piazolla}
I_n(\phi)=\frac{\mathcal{L}^n(\om)}{2}\left(2n^{n/2}-\dashint_{\om^-} f(x) \dx +\dashint_{\om^+} f(x) \dx \right).\end{align}
Putting \eqref{astor} and \eqref{piazolla} together gives \eqref{meanest}.
\end{proof}

\section{Two-state $f$}\label{Sec: two-state f}
In this section, we assume that $f$ takes only two values in $Q:=(-1,1)^n$, i.e., that $f=M\chi_\Omega$, where $M\in\R$,   $\Omega$ is a strict   subset of $Q$ that is open, and $\chi_\Omega$ is the characteristic function of $\Omega$ in $Q$.  Our aim is to show that, under a relatively mild regularity assumption on $\partial \om$, the functional 
\begin{align}\label{inomega}I_n(\varphi):=\int_{Q}|\nabla \varphi|^n+M\chi_{\Omega}\det \nabla \varphi \, dx \end{align}
is nonnegative on $W_0^{1,n}(Q,\R^n)$ if and only if $|M|\leq 2c_n$, where $c_n$ is the constant appearing in the pointwise Hadamard inequality \eqref{pwhad}. 
In fact, to show that $I_n\ge 0$ on $W_0^{1,n}(Q,\R^n)$ if $|M|\leq 2c_n$ and $\Omega$ is measurable is very easy because  
\begin{align}
I_n(\varphi)&=I_n(\varphi)-\underbrace{\frac M2\int_Q\det\nabla\varphi\dx}_{=0} \nonumber\\
&=
\int_\Omega |\nabla\varphi|^n +\frac M2\det\nabla\varphi\dx
+ \int_{Q\setminus\Omega} |\nabla\varphi|^n -\frac M2\det\nabla\varphi\dx
\end{align}
and both terms on the right-hand side are nonnegative in view of the pointwise Hadamard inequality.
The regularity assumption enables us, via a blow-up argument, to prove that the sequential weak lower semicontinuity of $I_n$ is equivalent to that of the special functional
\begin{align}\label{inplus}I_n^+(\varphi):=\int_{Q} |\nabla \varphi|^n + M \chi_{Q^+}\det \nabla \varphi \, dx,
\end{align}
which, in turn, holds if and only if $I_n^+\geq 0$ on $W_0^{1,n}(Q,\R^n)$, i.e. if and only if a particular form of (HIM) holds. Here, $Q^+:=\{x\in Q: x_n\ge 0\},$ and it is worth emphasising that the prefactor $M$ is exactly the same in both \eqref{inomega} and \eqref{inplus}.  Thus, by these steps, the nonnegativity of $I_n$ on $W_0^{1,n}(Q;\R^n)$ is equivalent to the nonnegativity of $I_n^+$ on the same space.  We shall see in Proposition \ref{inplusnonnegative} below that in fact
$$I_n^+(\varphi) \geq 0 \quad \forall \varphi \in W_0^{1,n}(\om,\R^n) \iff |M|\leq 2c_n,$$
a result whose proof serves to connect our work with the literature on quasiconvexity at the boundary \cite{bama}.    

It is convenient to divide the present section into two:  Subsection \ref{threepointone} deals exclusively with the nonnegativity of the special functional $I_n^+$, while  Subsection~\ref{threepointtwo} focuses on  $I_n$. 

\subsection{The nonnegativity of $I_n^+$}\label{threepointone}

The following result was proved by Mielke and Sprenger in  \cite[Thm.~5.1]{mielke-sprenger} for $n=2$.  Proposition~\ref{prop:n_dim_two_state} extends this result to any $n\ge 2$. If   $\Gamma$ represents the plane $\{x \in \R^n: \ x_n =0\}$ we  define 
\begin{align}
    W^{1,n}_{\Gamma}(Q^+,\R^n)=\{\phi\in W^{1,n}(Q^+;\R^n):\, \phi=0\, \text{ on } \partial Q^+\setminus\Gamma\}.
\end{align}

\begin{proposition}\label{prop:n_dim_two_state}

Let $n\ge 2$ and let  the functional $I_n^{++}$ be given by 
\begin{align}\label{n_dimensional_I+}
I_n^{++}(\varphi) :=\int_{Q^+} |\nabla \varphi|^n + c \det \nabla \varphi \dx. 
\end{align}
   Then $I^{++}(\varphi) \geq 0$ for all $\varphi \in W^{1,n}_{\Gamma}(Q^+,\R^n)$ if and only if $|c|\leq c_n$.
\end{proposition}

\begin{proof} That the stated condition is sufficient for the nonnegativity of $I_n^{++}$ is a direct consequence of inequality \eqref{pwhad}.  The `only if' part of the proposition can be proved by contraposition, as follows.  Firstly, by exchanging the first component $\varphi^0_1$ with $-\varphi^0_1$ if necessary, we may assume that $c\geq 0$.  Then we suppose that $c > c_n$, and find $\Phi \in W^{1,n}_{\Gamma}(Q^+,\R^n)$ such that $I_n^{++}(\Phi)<0$ as follows.  

Let $H^{\pm}=\{x \in \R^n: \pm x_n > 0\}$.  Let $e_n$ be the $n^{\textrm{th}}$ canonical basis vector, let $\delta>0$ and set $a=-\delta e_n$. Define for $x \in H^+$ the function 
$$\varphi^0(x)=A\frac{x-a}{|x-a|^2},$$
where $A$ is a fixed $n \times n$ matrix to be chosen shortly.  Then 
\begin{align*}
    \nabla \varphi^0(x) = \frac{A}{|x-a|^2}(\1 - 2 \nu(x) \otimes \nu(x))
\end{align*}
provided $\nu(x)=\frac{x-a}{|x-a|}$.  The matrix $P:=(\1 - 2 \nu(x) \otimes \nu(x))$ obeys $P^2=\1$ and $\det P=-1$.   Choosing $A$ so that it belongs to the set $CO_{+}(n)$ of conformal matrices
gives 
\begin{align*}
- c_n \det \nabla \varphi^0 =  |\nabla \varphi^{0}|^n,     
\end{align*}
and we see that the integrand of $I_n^+(\varphi^0)$ obeys
\begin{align*}
F(\nabla \varphi^0) & = |\nabla \varphi^0|^n + c \det \nabla \varphi^0 \\
& = -\left(\frac{c}{c_n}-1\right) |\nabla \varphi^0|^n, 
\end{align*}
which, other than the fact that $\varphi^0$ fails to have compact support in $Q^+$, would suffice to complete the proof.  The remaining step remedies this using a cut-off argument and by allowing the singularity at $a=-\delta e_n$ to approach the boundary $\Gamma$ of $H^+$.  

Let $R_0>\delta$ be constant, and let $\eta: \ [0,\infty) \to [0,1]$ be a smooth cut-off function such that $\eta(s) = 1$ for $0 \leq s \leq R_0$, $\eta(s)=0$ for $s\geq 2R_0$ and $|\eta'(s)| \leq \frac{1}{R_0}$.  Let $\varphi(x)=\eta(|x|) \varphi^0(x)$ and note that $\nabla \varphi(x) = \eta(|x|) \nabla \varphi^0(x) + \eta'(|x|) \varphi^0(x) \otimes \frac{x}{|x|}$ for $x \in H^+$.  
In $B(0,R_0)$,  $\varphi$ coincides with $\varphi^0$, so 
\begin{align*}
    \int_{B(0,R_0) \cap H^+} F(\nabla \varphi) \dx & =  -\left(\frac{c}{c_n}-1\right) \int_{B(0,R_0) \cap H^+} |\nabla \varphi^0|^2 \dx \\& = -\left(\frac{c}{c_n}-1\right) \int_{B(0,R_0) \cap H^+} \frac{|A|^n}{|x-a|^{2n}} \dx.
\end{align*}
Now 
\begin{align*}
    \int_{B(0,R_0) \cap H^+} \frac{dx}{|x-a|^{2n}} & 
    \geq C \int_0^{R_0} \frac{\rho^{n-1}}{(\rho+\delta)^{2n}} \drho \\
    &= C \int_{\delta}^{R_0}  \left(\frac{\rho}{\rho+\delta}\right)^{n-1} \frac{1}{(\rho+\delta)^{n+1}} \, d\rho \\
      & \geq 
    2^{1-n} C \int_\delta^{R_0} \frac{1}{(\rho+\delta)^{n+1}} \drho \\
    & = \frac{2^{1-2n}C}{n}\left(\delta^{-n} - 2^n(R_0 +\delta)^{-n}\right),
\end{align*}
where $C$ is a constant depending only on the angular part of the integral, and is hence dependent only on $n$.  Let $\omega=(B_{2R}\setminus B_R) \cap H^+$.  The other contributor to $I^+(\varphi)$ is such that, as long as $R_0 \geq 2 \delta$, say, then
\begin{align*}
    \left|\int_{\omega} F(\nabla \varphi) \dx \right| \dx & \leq C'\int_{\om} |\nabla \varphi|^n \dx \\
    & \leq C' \int_{\omega} |\nabla \varphi^0|^n + \frac{|\varphi^0|^n}{R_0^n} \dx
    \\ & \leq C'\left( \int_{\omega} \frac{dx}{|x-a|^{2n}} + \int_{\omega}\frac{\dx}{R_0^n |x-a|^{n}}   \right) \\
    & \leq C'\left((R_0-\delta)^{-n}-(2R_0-\delta)^{-n}\right),
    \end{align*}
where $C'$ is a constant depending only on $n$ and which may change from line to line.  

Hence, for positive constants $a_1$ and $a_2$ depending only on $n$, we have 
\begin{align*}
    I_n^{++}(\varphi) \leq -\left(\frac{c}{c_n}-1\right)a_1 \left(\delta^{-n} -2^n(R_0+\delta)^{-n}\right) + a_2 \left((R_0-\delta)^{-n} - (2R_0 -\delta)^{-n}\right),
\end{align*}
the right-hand side of which is, for fixed $c>c_n$, suitably small $\delta$ and suitably large $R_0$, negative.  The function
$$\Phi(x):=\frac{1}{2R_0} \, \varphi(2R_0 x)$$
belongs to $W^{1,n}_{\Gamma}(Q^+,\R^n)$ and it obeys $I^{++}(\Phi)<0$.
\end{proof}
Hence, taking $F\in\R^{n\times n}$,  $F\mapsto |F|^n+c\det F$ is quasiconvex at the boundary for any normal if and only if $|c|\le c_n$, i.e., only if  it is pointwise nonnegative. 
This extends the result from \cite{mielke-sprenger} to any dimension.

We can now prove the main result of this subsection.

\begin{proposition}
\label{inplusnonnegative}
The functional $I_n^+$ given by \eqref{inplus}  is nonnegative on $W^{1,n}_0(Q;\R^n)$ if and only if $|M|\le 2c_n$. 
\end{proposition}

\begin{proof}

 Given \referee  $\varphi\in W_{\Gamma}^{1,n}(Q^+;\R^n)$ \EEE let $\psi\in W^{1,n}_0(Q;\R^n)$  be defined as follows:
 \begin{align}\label{reflection}
   \psi(x)=\begin{cases}
   \varphi(x) &\text{ if } x\in Q^+\ ,\\
   \varphi(x_1,x_2,\ldots, -x_n) & \text{ if } x\in Q\setminus Q^+\ .
   \end{cases}
 \end{align}
  Notice that 
\begin{eqnarray*}
&&\int_{Q^+}|\nabla\psi(x)|^n \dx=\int_{Q\setminus Q^+}|\nabla\psi(x)|^n \dx , \\
&&\int_{Q^+}\det\nabla\psi(x) \dx=-\int_{Q\setminus Q^+}\det\nabla\psi(x) \dx , \\
&&\int_{Q}\det\nabla\psi(x) \dx=0.
\end{eqnarray*}
 We calculate 
 \referee
 \begin{multline*}
I_n^+(\psi)=\int_{Q}|\nabla\psi|^n \dx +M\int_{Q^+}\det\nabla\psi \dx \\
= \int_{Q^+}  |\nabla\psi|^n \dx +\frac{M}{2}  \int_{Q^+}\det\nabla\psi \dx +\int_{Q\setminus Q^+} |\nabla\psi|^n \dx -\frac{M}{2}  \int_{Q\setminus Q^+}\det\nabla\psi \dx\\
=2\Big( \int_{Q^+}|\nabla\varphi|^n \dx +\frac{M}{2}  \int_{Q^+}\det\nabla\varphi \dx\Big).
\end{multline*}
\EEE
 It follows from  Proposition~\ref{prop:n_dim_two_state},  however, that for every $M\in\R$ such that $|M|>2c_n$ there is 
 $
 \varphi_M\in W_{\Gamma}^{1,n}(Q^+;\R^n)$  satisfying 
  $$
  \int_{Q^+}|\nabla\varphi_M|^n \dx +\frac{M}{2}  \int_{Q^+}\det\nabla\varphi_M \dx<0, $$
 i.e., $I_n^+(\psi_M)<0$ with $\psi_M$ defined by means of   $\varphi_M$   in  \eqref{reflection}. Sufficiency was already shown.
 \end{proof}
 
 \begin{remark}\label{rem:qcb}
 The set $Q^+$  is called the standard boundary domain in \cite{bama}. The integrand of  $I_n^{++}$ is  quasiconvex at the boundary at zero  with respect to the normal $-e_n=(0,\ldots, -1)$ if and only if $|M|\le c_n$. Quasiconvexity at the boundary forms a necessary condition for minimizers in elasticity. In fact, this property does not depend on the integration domain and $Q^+$ can be replaced by  a half-ball $D_{-e_n}=\{x\in B(0,1):\, x_n>0\}$, where $B(0,1) $ is the unit ball centred at the origin  defined by  the Frobenius norm, for instance. The set $W_{\Gamma}^{1,n}(Q^+;\R^n)$ must then be replaced by 
 $$\{\varphi\in W^{1,n}(B;\R^n):\, \varphi=0 \text{ on } \partial D_{-e_n}\setminus\{x_n=0\}\}. $$
 We refer to \cite{bama, mielke-sprenger} for details. 
 It is easy to see that the normal $-e_n$ above can be replaced by any other unit vector which just means that the function  
 $F\mapsto |F|^n + M\det F$ is quasiconvex at the boundary at zero  with respect to any normal if and only if $|M|\le c_n$.
 \end{remark}

\subsection{Nonnegativity of $I_n$}\label{threepointtwo}
In this subsection, we assume  that $\Omega\subset Q$ is an open  set such that there is a point
$x_0\in\partial\Omega\setminus\partial Q$ at which the unit outer normal  to 
$\partial\Omega$ exists.

\begin{proposition}\label{islandproblem}
Let $I_n$ be as in \eqref{inomega}. If $|M|>2c_n$ then there is $\varphi\in W^{1,n}_0(Q;\R^n)$ such that $I_n(\varphi)<0$.  \referee In particular, $I_n(\varphi) \geq 0$ for all $\varphi \in W_0^{1,n}(Q;\R^n)$ if and only if $|M| \leq 2c_n$. \EEE
\end{proposition}
\begin{proof}
We  assume without loss of generality  that $\nu\in\R^n$  is  the outer unit normal to $\partial\Omega$ at $x_0=0$.  Consider $r>0$ so small that $B(0,r)\subset Q$, take $\psi\in W^{1,n}_0(B(0,r);\R^n)$, extend it by zero to the whole $\R^n$ and define 
$\varphi_k(x)=\psi(kx)$ for every $k\in\N$. Moreover, let $D_\nu=\{x\in B(0,r): x\cdot\nu>0\}$.
\begin{align}\label{Jn}
I_n(\varphi_k)&=\int_Q|\nabla \varphi_k|^n\, \dx +M\int_{\Omega}\det\nabla \varphi_k\dx\nonumber\\
&= \int_{B(0,r)} k^n|\nabla\psi(kx)|^n\dx+ Mk^n\int_{\Omega\cap B(0,r)}\det\nabla \psi(kx)\dx\nonumber\\
&=\int_{B(0,r)}|\nabla \psi|^n\, \dx +M\int_{\{kx: x\in \Omega\cap B(0,r)\}}\det\nabla \psi\dx\nonumber\\
&\to\int_{B(0,r)}|\nabla \psi|^n\, \dx +M\int_{D_\nu}\det\nabla \varphi\dx.
\end{align}

If $|M|>2c_n$ we can find $\psi\in W^{1,n}_0(B(0,r);\R^n)$ such that the last term is negative; cf.~Remark~\ref{rem:qcb}.   This implies that there is $\varphi_k\in W^{1,n}(Q;\R^n)$ such that 
$I_n(\varphi_k)<0$.
\end{proof}

\section{Insulation and point-contact problems}\label{insulation}
\renewcommand\thefigure{\thesection.\arabic{figure}}

The purpose of this section is to demonstrate concretely that, in the planar case $n=2$, there are functionals of the form
$$I(\varphi)=\int_{\om}|\nabla \varphi|^2 + f(x) \det \nabla \varphi(x) \dx $$
for which (HIM) holds and where the total variation in $f$
$$\delta f:={\rm ess}\sup\{|f(x)-f(y)|:\ x,y \in \om\}$$
obeys $\delta f >4$.  When $\delta f \leq 4$, (HIM) holds straightforwardly, so the arguments in Sections \ref{insuproper} and \ref{-4040} are necessarily much more delicate.  We infer from both examples that there is a subtle interplay between the geometry of the domain, here understood to be the regions on which $f$ is constant, and the values taken by $f$ there.  We distinguish the problems we consider by the manner in which the regions where $f$ is constant interact: in a pure `insulation problem' the regions $f^{-1}(\pm c)$ lie at a positive distance from one another and are separated by $f^{-1}(0)$, whereas in a point-contact insulation problem the regions $f^{-1}(\pm c)$ meet in a point and are otherwise separate.

Section \ref{insuproper} concerns the pure insulation problem, in which we divide a rectangular domain $\om$ into three adjacent regions $\om_l,\om_m$ and $\om_r$, say, and set
\begin{align}\label{threestate} f:=-c\chi_{\om_l} + c \chi_{\om_r},
\end{align}
where $c$ is constant.  In order that $\delta f>4$ we must take $c>2$, but, as we know from Section~\ref{Sec: two-state f}, a jump of $2c>4$ would then be incompatible with (HIM) were the regions $\om_l$ and $\om_r$ to share a common boundary.  Thus the role of the middle region $\om_m$ is to `insulate' $\om_l$ from $\om_r$, and the force of Proposition \ref{sum:4+} is that, for the particular choice of $\om$, $\om_l$, $\om_m$ and $\om_r$ described there, there are $c>2$ such that (HIM) holds.  We also remark in Proposition \ref{thinmiddle} that, given a `three-state' $f$ as described in \eqref{threestate}, the `insulation strip' cannot be arbitrarily small.  

Section \ref{-4040} describes in detail a point-contact insulation problem.   We produce a three-state $f$ with $\delta f=2\sqrt{8}$ and for which (HIM) holds. 
The result is optimal in that if the geometry of the sets $f^{-1}(0)$ and $f^{-1}(\pm c)$ is retained but the constant c is chosen larger in modulus than $\sqrt{8}$, then the associated functional can be made negative and (HIM) fails. 
See Section  \ref{-4040} for details. 

\subsection{The pure insulation problem}\label{insuproper}


We now focus on the domain 
$$\om:=R_{-2} \cup R_{-1} \cup R_{1} \cup R_2$$ 
formed of four rectangles arranged in a row, as shown in Figure \ref{pic:pureinsulation}.  
\begin{figure}[H]
\centering
\begin{minipage}{0.69\textwidth}
\centering
\begin{tikzpicture}[scale=0.8]
\node (A) at (-4,-2) {}; 
\node[right=2 of A.center] (B) {};
\node[right=4 of A.center] (O) {};
\node[right=4 of B.center] (C) {};
\node[right=2 of C.center] (D) {};

\node[above=4 of A.center] (A') {};
\node[above=4 of B.center] (B') {};
\node[above=4 of O.center] (O') {};
\node[above=4 of C.center] (C') {};
\node[above=4 of D.center] (D') {};

\node[right=2 of B.center] (BC) {};
\node[right=2 of B'.center] (B'C') {};

\filldraw[thick, top color=gray!30!,bottom color=gray!30!] (A.center) rectangle node{$-c$} (B'.center);
\filldraw[thick, top color=gray!30!,bottom color=gray!30!] (C.center) rectangle node{$+c$} (D'.center);

\node (R1) at ($(B)!0.5!(B'C')$) {0};    
\node (R-1) at ($(BC)!0.5!(C')$) {0};    

\node [below right = 0.1 of A'.center] {$R_{-2}$};
\node [below right = 0.1 of B'.center] {$R_{-1}$};
\node [below right = 0.1 of O'.center] {$R_{1}$};
\node [below right = 0.1 of C'.center] {$R_{2}$};

\draw[thick] (B.center) -- (C.center);
\draw[thick] (B'.center) -- (C'.center);
\draw[thick] (BC.center) -- (B'C'.center);
\end{tikzpicture}
\end{minipage}
\begin{minipage}{0.29\textwidth}
\begin{align*}
R_{-2}  &=[\scalebox{0.9}{$-1,-\frac{1}{2}$}]\times [\scalebox{0.9}{$-\frac{1}{2  }$},\scalebox{0.9}{$\frac{1}{2}$}], \\
     R_{-1}  &=[-\scalebox{0.9}{$\frac{1}{2}$},0]\times [\scalebox{0.9}{$-\frac{1}{2  }$},\scalebox{0.9}{$\frac{1}{2}$}], \\
     R_1  &=[0,\scalebox{0.9}{$\frac{1}{2}$}]\times [\scalebox{0.9}{$-\frac{1}{2}$},\scalebox{0.9}{$\frac{1}{2}$}], \\
     R_2  &=[\scalebox{0.9}{$\frac{1}{2}$},1]\times [\scalebox{0.9}{$-\frac{1}{2}$},\scalebox{0.9}{$\frac{1}{2}$}].
\end{align*}
\end{minipage}
\caption{Distribution of rectangles. 
}  \label{pic:pureinsulation}
\end{figure}
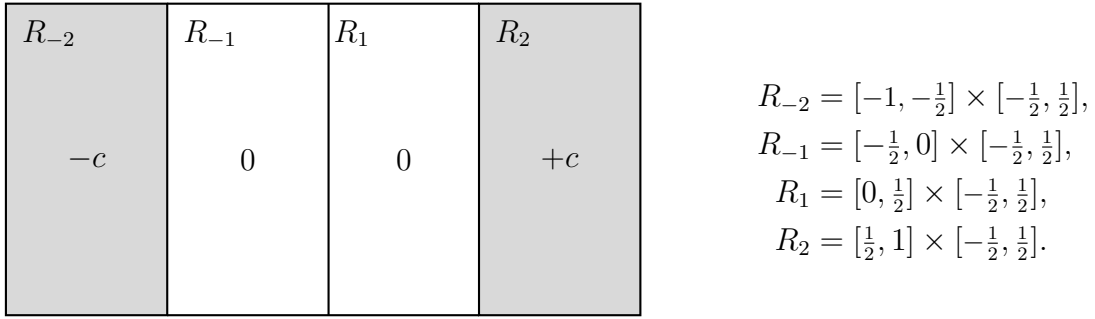

In the notation introduced above, we therefore have
\begin{align*}
    \om_l  =R_{-2}, \quad
    \om_m  =R_{-1}\cup R_1, \quad
    \om_r  = R_2,
\end{align*}
with the middle region $\om_m$ playing the role of an `insulating' layer. 

Let $c>0$ be constant, define the piecewise constant function $f$ by
\begin{align}\label{f_four_square} f(x):=-c\chi_{_{R_{-2}}} + c \chi_{_{R_{2}}}, \end{align}
and form the functional
\begin{align}\label{I_four_square} I(\varphi,\INS{-c}{0}{0}{c} \,)& := \int_{\om} |\nabla \varphi|^2+f(x) \det \nabla \varphi(x) \dx. \\ & = 
\nonumber \int_{\om}|\nabla \varphi|^2 \dx + c\int_{R_2} \det \nabla \varphi \dx - c \int_{R_{-2}} \det \nabla \varphi \dx.
\end{align}

It is useful to write the functional $I(\varphi,\INS{-c}{0}{0}{c} \,)$ in terms of the odd and even parts of $\varphi$, which we now do.

\begin{lemma} Let $\varphi$ belong to $W^{1,2}_0(\om,\R^2)$ and let $f$ and $g$ be, respectively, the even and odd parts of $\varphi$.  Then 
\begin{multline}\label{corpus}  \frac{1}{2}I(\varphi,\INS{-c }{ 0 }{ 0 }{ c } \,)  \\ =  \int_{R_1} |\nabla f|^2 + |\nabla g|^2 \dx +    \int_{R_2} |\nabla g|^2 + |\nabla f|^2 + c   \, \cof\nabla f \cdot \nabla g \, \dx.
\end{multline}
In particular, if we set
\begin{align*}
\Gamma_0: =\{\scalebox{0.9}{$0$}\}\times [-1/2,1/2], \quad
\Gamma_1: =\{\scalebox{0.9}{$\frac{1}{2}$}\}\times [-1/2,1/2],
\end{align*} 
with $f_0=f\arrowvert_{_{\Gamma_0}}$, 
$f_1=f\arrowvert_{_{\Gamma_1}}$, then the estimate
\begin{align}\label{AcesUp}
 \frac{1}{2}I(\varphi,\INS{-c}{0}{ 0}{c} \,) \geq \int_{R_1} |\nabla F_1|^2 \dx +  \left(1-\frac{c^2}{4}\right) \int_{R_2} |\nabla F_2|^2 \dx
\end{align}
holds, where each $F_j$ is harmonic on $R_j$ and obeys $F_j\arrowvert_{_{\Gamma_1}}=f_1$ for $j=1,2$.  
\end{lemma}

\begin{proof} 
By approximating a general $\varphi$ in $W^{1,2}_0(Q,\R^2)$ with smooth, compactly supported maps, we may assume in the following that $\varphi$ belongs to $C_c^{\infty}(Q;\R^2)$.  In particular, the traces $f_0=\varphi_{_{\Gamma_0}}$ and $f_1=\varphi_{_{\Gamma_1}}$ are well defined.  Let $\phi_1$ be the harmonic extension of $\varphi_{_{\partial R_1}}$, with a similar definition for $\phi_2$. Then
\begin{align*}
\int_{R_1}|\nabla \varphi|^2 \dx & \geq \int_{R_1}|\nabla \phi_1|^2 \dx     \end{align*}
and 
\begin{align*}
\int_{R_2}|\nabla \varphi|^2 \dx & \geq \int_{R_2}|\nabla \phi_2|^2 \dx.     \end{align*}
Noting that $\det \nabla \varphi$ is a null Lagrangian, we also have 
\begin{align*}
    \int_{R_j} \det \nabla \varphi \dx = \int_{R_j} \det \nabla \phi_j \dx \quad \quad j=1,2.
\end{align*}
Hence, 
\begin{align*}\int_{R_1\cup R_2}|\nabla \varphi|^2 + c\chi_{_{R_2}}\det \nabla \varphi \dx & \geq \int_{R_1}|\nabla \phi_1|^2 \dx + \int_{R_2}|\nabla \phi_2|^2 +c \, \det \nabla \phi_2 \dx. \end{align*}
It follows that, in bounding the functional $I(\varphi,\INS{-c}{0}{ 0}{c} \,)$ below,  we may assume without loss of generality that $\varphi$ is harmonic on each of $R_1$ and $R_2$.  Henceforth, we replace $\phi_1$ and $\phi_2$ with $\varphi$ and assume that the last of these is harmonic in each of $R_1$ and $R_2$.

To prove \eqref{corpus}, write  $\varphi=f+g$ and use the facts that $\nabla f$ and $\nabla g$ are, respectively, odd and even functions, to see that
\begin{align}\label{dir:let}
\int_{\om} |\nabla \varphi|^2 \dx = 2\int_{R_1}|\nabla f|^2+|\nabla g|^2 \dx +2\int_{R_2}|\nabla f|^2+|\nabla g|^2 \dx .\end{align}
Making use of the expansion $$\det \nabla \varphi= \det \nabla f + \det \nabla g + \cof \nabla f \cdot \nabla g,$$
the determinant terms can be handled similarly, giving
\begin{align}\label{det:erm}
\int_{R_2}\det \nabla \varphi \dx - \int_{R_{-2}} \det \nabla \varphi \dx = 2 \int_{R_2} \cof \nabla f \cdot \nabla g \dx.
\end{align}
To obtain \eqref{corpus}, combine \eqref{dir:let} and \eqref{det:erm}.  

Now, the rightmost integrand of \eqref{corpus} obeys, pointwise\footnote{For example, consider the function $h(A,B):=|A|^2+|B|^2 +c \, \cof A \cdot B$ defined on pairs of $2\times 2$ matrices, freeze $A$, say, and minimize $h(A,B)$ over $B$, which is easily done since $B \mapsto h(A,B)$ is strongly convex.} 
\begin{align*}
 |\nabla f|^2+|\nabla g|^2 + c \,  \cof \nabla f \cdot \nabla g & \geq \left(1-\frac{c^2}{4}\right)|\nabla f|^2,    
\end{align*}
from which the estimate
\begin{align}\label{cristi}  \frac{1}{2}I(\varphi,\INS{-c }{ 0 }{ 0 }{ c } \,)  \geq  \int_{R_1} |\nabla f|^2 + |\nabla g|^2 \dx +\left(1-\frac{c^2}{4}\right)\int_{R_2}|\nabla f|^2\dx.
\end{align}
is obtained.  Finally, \eqref{AcesUp} follows by dropping the term $\int_{R_1}|\nabla g|^2\dx$ and by adopting the notation $F_j\arrowvert_{_{\Gamma_1}}=f_1$ for $j=1,2$.  
\end{proof}

The two Dirichlet energies appearing in \eqref{AcesUp} are linked by the condition that 
\begin{align*}
F_j \arrowvert_{\Gamma_1}=f_1 \quad \quad j=1,2,
\end{align*} and we claim that this link is enough to make the two quantities comparable in the sense that 
\begin{align*} \int_{R_1} |\nabla F_1|^2 \dx \geq \gamma \int_{R_2} |\nabla F_2|^2 \dx
\end{align*}
for some $\gamma>0$ that is independent of $f_1$.  To prove the claim we begin by setting out two auxiliary Dirichlet problems on the unit ball $B(0,1) \subset \R^2$ as follows.   Suppose that $h: [0,\pi] \to \R^2$ is such that
\begin{itemize}\item[(H1)] $h$ has compact support in the interval $\omega:=[\pi/4,3\pi/4]$ and $h$ belongs to $H^1([0,\pi])$ and 
\item[(H2)] $h(\theta)=h(\pi-\theta)$ for $0 \leq \theta \leq \pi$.  
\end{itemize}

Define $h^{\pm} \in H^1(\R;\R^2)$ by 
\begin{align}\label{hpm} h^{\pm}(\theta)=\left\{ \begin{array}{l l} \ \  h(\theta) & \textrm{if } 0 \leq \theta \leq \pi \\
\pm h(\theta-\pi) & \textrm{if } \pi \leq \theta \leq 2\pi, \end{array}\right.
\end{align}
extended periodically to $\R$.  We define the lift $\tilde{h}^{\pm}: \mathbb{S}^1 \to \R^2$ via 
\begin{align*}
\tilde{h}^{\pm}(x):=h^{\pm}(\theta) \quad \quad x=(\cos \theta, \sin \theta).
\end{align*}

\vspace{2mm}
\noindent \textbf{Problem $\oplus$:}  Minimize $\D(w):=\int_B |\nabla w|^2 \dx$ in 
$\{w \in H^1(B;\R^2): \ w\arrowvert_{\partial B}= \tilde{h}^+\}.$

\vspace{2mm}
\noindent \textbf{Problem $\ominus$:}  Minimize $\D(z):=\int_B |\nabla z|^2 \dx$ in 
$\{z \in H^1(B;\R^2): \ z\arrowvert_{\partial B} = \tilde{h}^-\}.$

\vspace{2mm}
Let $W$ and $Z$ solve problems $\oplus$ and $\ominus$ respectively.  By an application of Schwarz's formula, we can express the components of $W$ and $Z$ in the form 
\begin{align*} W_j(z) & =\re(2H_j^+(z)-H_j^+(0)), \\
Z_j(z) & = \re(2H_j^{-}(z)-H_j^-(0))
\end{align*}
for $j=1,2$, where 
\begin{align*}H^{\pm}(z):=
\frac{1}{2\pi}\int_0^{2\pi} \frac{h^{\pm}(\theta)}{1-z e^{-i\theta}} \dtheta. 
\end{align*}
Letting 
$$h^{+}(\theta)=\frac{a^+_0}{2}+\sum_{k=1}^\infty a^+_k \cos(k \theta) + b^+_k \sin(k \theta)$$
and 
$$h^{-}(\theta)=\sum_{k=1}^\infty a^-_k \cos(k \theta) + b^-_k \sin(k \theta)$$
be the Fourier series representations of $h^+$ and $h^-$ respectively, a direct calculation shows that 
$$W(R,\theta)=\frac{a_0^+}{2}+\sum_{k=1}^{\infty} (a^+_k \cos(k \theta) + b^+_k \sin(k \theta))R^k,$$
and similarly for $Z$.  The Dirichlet energies $\D(W,B)$ and $\D(Z,B)$ are then
\begin{align}\label{whalf} \D(W,B) & = \pi \sum_{k=1}^\infty k({|a^+_k|}^2+{|b^+_k|}^2) \\
\label{zhalf} \D(Z,B) & =  \pi \sum_{k=1}^\infty k({|a^-_k|}^2+{|b^-_k|}^2),
\end{align}
which we recognise as being proportional to the (squared) $H^{1/2}-$norms of $\tilde{h}^+$ and $\tilde{h}^-$ respectively.  The following lemma can be deduced from its continuous Fourier transform counterpart (see, for instance, \cite[Prop. 3.4]{NPV2012} in the case that $s=\frac{1}{2}$.) We include the proof for completeness.

\begin{lemma} \label{hitchlem}Let $h(\theta)$ and its $2\pi-$periodic extension to $\R$ be represented by the Fourier series 
\begin{align*} h(\theta) = \frac{A_0}{2}+\sum_{k=1}^\infty A_k\cos(k \theta) + B_k \sin( k \theta) \quad \quad \quad 0 \leq \theta \leq 2 \pi.
\end{align*}
 Then provided the right-hand side of \eqref{thehitch} is finite,
\begin{align}\label{thehitch}[h]_{_{\frac{1}{2}}}^2:= \int_{0}^{2\pi} \int_{\R} \frac{|h(\theta)-h(\phi)|^2}{(\theta-\phi)^2} \dtheta \dphi = 2 \pi^2 \sum_{k=1} k(A_k^2+B_k^2).
\end{align}
\end{lemma}

\begin{proof} Let
$    G(\theta,\phi) = \frac{h(\theta)-h(\phi)}{\theta-\phi}$
and note that, by a simple change of variables, 
\begin{align}
 \nonumber \int_0^{2\pi} \int_{\R} G(\theta,\phi)^2 \dtheta \dphi & =  \int_0^{2\pi} \int_{\R} G(z+\phi,\phi)^2 \dz \dphi \\ \label{ftt1} & = \int_0^{2\pi} \int_{\R} \frac{(h(z+\phi)-h(\phi))^2}{z^2} \dz \dphi.
\end{align}
We focus on showing that 
\begin{align}
\label{ftt2} \int_{\R}  \int_0^{2\pi} \frac{(h(z+\phi)-h(\phi))^2}{z^2} \dphi \dz & = 2 \pi^2 \sum_{k=1} k(A_k^2+B_k^2)   
\end{align}
under the assumption that its right-hand side is finite.  Once this is done, Fubini-Tonelli will guarantee that the integrals in \eqref{ftt1} and \eqref{ftt2} coincide, from which the proof is easily concluded.

Now, a short calculation shows that  
\begin{align*}\frac{h(\phi+z)-h(\phi)}{z} & = \sum_{k=1}^\infty 
\frac{U_k(z)}{z}\cos(k\phi) + \frac{V_k(z)}{z} \sin(k \phi) 
    \end{align*}
    where
    \begin{align*}
    U_k(z)& =A_k (\cos(kz)-1)+B_k \sin(kz), \\
    V_k(z) & = B_k (\cos(kz)-1) - A_k \sin(kz).
    \end{align*}
Since the right-hand side of \eqref{ftt2} is finite, and since each of $|U_k(z)|$ and $|V_k(z)|$ can be bounded above by a fixed multiple of $|A_k|+|B_k|$, it is clear that for a.e.\@ fixed $z$ the function  $\phi \mapsto G(z+\phi,\phi)=\sum_{k=1}^\infty \frac{U_k(z)}{z}\cos(k\phi)+ \frac{V_k(z)}{z}\sin(k\phi)$ belongs to $L^2(0,2\pi)$, and moreover, by Parseval's formula, that 
\begin{align}
   \label{ftt3}  \int_0^{2\pi} G(z+\phi,\phi)^2 \dphi & = \pi \sum_{k=1}^\infty \frac{U_k^2(z)+V_k^2(z)}{z^2} \\
  \label{ftt4}  & = 2 \pi \sum_{k=1}^\infty (A_k^2+B_k^2)\left(\frac{1-\cos(kz)}{z^2}\right). 
\end{align}
A short calculation shows that for each $k \in \N$
\begin{align*}
    \int_{\R} \frac{1-\cos(kz)}{z^2} \dz = \pi k,
\end{align*}
and an application of the monotone convergence theorem then yields 
\begin{align} \label{ftt5}
   \int_{\R} 2 \pi \sum_{k=1}^\infty (A_k^2+B_k^2)\left(\frac{1-\cos(kz)}{z^2}\right) \dz & = 2\pi^2 \sum_{k=1}^{\infty}k(A_k^2+B_k^2).
\end{align}
Putting \eqref{ftt3}, \eqref{ftt4} and \eqref{ftt5} together gives \eqref{ftt2}, as desired.  
\end{proof}

Applying \eqref{thehitch} to  \eqref{whalf} and \eqref{zhalf}, we see that
\begin{align} \label{faurep}\D(W,B)& = \frac{1}{\pi} [h^+]_{_{\frac{1}{2}}}^2, \quad \mathrm{and} \\
\label{faurem}\D(Z,B)& = \frac{1}{\pi} [h^-]_{_{\frac{1}{2}}}^2.
\end{align}
It turns out that $[h^+]_{_{\frac{1}{2}}}$ and $[h^-]_{_{\frac{1}{2}}}^2$ are `comparable quantities' in the following sense.
\begin{proposition}\label{comphpm} Let $h$ satisfy (H1) and (H2).  Then there is $\gamma_0>0$ independent of $h$ such that 
\begin{align}\label{sonofhitch}
[h^+]_{_{\frac{1}{2}}}^2 \geq \gamma_0 [h^-]_{_{\frac{1}{2}}}^2.
\end{align}
\end{proposition}
\begin{proof}
In the following, we let 
\begin{align*}
    \delta(\theta,\phi) &:=h(\theta)-h(\phi) \\
    \sigma(\theta,\phi) & :=h(\theta)+h(\phi)
\end{align*}
and $T_n(\theta,\phi)=(\theta-\phi - n\pi)^{-2}$ for all $n \in \N$ and $\theta,\phi \in [0,\pi]$. Using Lemma \ref{hitchlem} and the definitions of $h^{\pm}$, a calculation shows that
\begin{align*}
    \frac{1}{2}[h^+]_{\frac{1}{2}}^2 & = \int_0^\pi \int_0^\pi  \delta^2(\theta,\phi)T_0(\theta,\varphi) \dtheta\dphi + 2 \underbrace{\sum_{k=1}^\infty \int_0^\pi \int_0^\pi \delta^2(\theta,\phi) (T_{2k}(\theta,\phi)+T_{2k-1}(\theta,\phi)) \dtheta\dphi}_{S_+} \\
    \frac{1}{2}[h^-]_{\frac{1}{2}}^2 & = \int_0^\pi \int_0^\pi  \delta^2(\theta,\phi)T_0(\theta,\varphi) \dtheta\dphi \, + & \\
    &  \quad \quad \quad \quad \quad \quad  + 2\underbrace{\sum_{k=1}^\infty \int_0^\pi \int_0^\pi \delta^2(\theta,\phi)T_{2k}(\theta,\phi) \dtheta\dphi + 
    \int_0^\pi \int_0^\pi \sigma^2(\theta,\phi)T_{2k-1}(\theta,\phi) \dtheta\dphi}_{S_{-}},
\end{align*}
Recall that $\textrm{spt}\,h\subset \omega=[\pi/4,3\pi/4]$, denote by $\tilde \omega$ the complement of $\omega$ in $[0,\pi]$ and consider
\begin{align*}
   S_+ & \geq \sum_{k=1}^\infty \int_{\omega} \int_{\tilde \omega} \delta^2(\theta,\phi) T_{2k}(\theta,\phi) \dtheta\dphi +\int_{\tilde\omega} \int_{\omega} \delta^2(\theta,\phi)T_{2k}(\theta,\phi) \dtheta\dphi + \\ 
   & +\int_{\omega} \int_{\tilde \omega} \delta^2(\theta,\phi)T_{2k-1}(\theta,\phi) \dtheta\dphi +\int_{\tilde\omega} \int_{\omega} \delta^2(\theta,\phi)T_{2k-1}(\theta,\phi) \dtheta\dphi,
\end{align*}
which uses the fact that $(\tilde{\omega} \times \omega) \cup (\omega \times \tilde{\omega}) \subset [0,1]^2$ and that the two sets forming the union are disjoint.  Notice that 
\begin{displaymath}
(h(\theta)-h(\phi))^2=\left\{\begin{array}{l l}h^2(\theta) & \textrm{if } (\theta,\phi)\in \omega \times \tilde \omega \\
h^2(\phi) & \textrm{if } (\theta,\phi)\in \tilde \omega \times \omega. \end{array}\right.
\end{displaymath}
Hence the summands simplify to 
\begin{align*}\int_\omega\int_{\tilde \omega} h^2(\theta)T_{2k}(\theta,\phi)\dtheta \dphi + & \int_{\tilde\omega} \int_{\omega} h^2(\phi) T_{2k}(\theta,\phi) \dtheta \dphi + \\   & +\int_\omega\int_{\tilde \omega} h^2(\theta)T_{2k-1}(\theta,\phi)\dtheta \dphi + \int_{\tilde\omega} \int_{\omega} h^2(\phi) T_{2k-1}(\theta,\phi) \dtheta \dphi,
\end{align*}
and, taking each in term in turn, we calculate
\begin{align*}
    \int_\omega\int_{\tilde \omega} h^2(\theta)T_{2k}(\theta,\phi)\dtheta \dphi & = \int_{\omega}h^2(\phi) K_2(\phi;k) \dphi, \\
    \int_{\tilde\omega} \int_{\omega} h^2(\phi) T_{2k}(\theta,\phi) \dtheta \dphi  & =  \int_{\omega}h^2(\phi) K_1(\phi;k) \dphi, \\
    \int_\omega\int_{\tilde \omega} h^2(\theta)T_{2k-1}(\theta,\phi)\dtheta \dphi & =  \int_{\omega}h^2(\phi) K_4(\phi;k) \dphi,\\
    \int_{\tilde\omega} \int_{\omega} h^2(\phi) T_{2k-1}(\theta,\phi) \dtheta \dphi & =  \int_{\omega}h^2(\phi) K_3(\phi;k) \dphi.
\end{align*}
Here, the kernels $K_i(\phi,k)$, $i=1,\dots,4$, are given by 
\begin{align*}
    K_1(\phi,k) & = K(\phi,-2k-1/4)-K(\phi,- 2k) +K(\phi,-2k-1)-K(\phi,-2k-3/4) \\
    K_2(\phi,k) & = K(\phi,2k-1/4) -K(\phi,2k)+K(\phi,2k-1)-K(\phi,2k-3/4) \\
    K_3(\phi,k) & = K(\phi,-2k+3/4))-K(\phi,-2k+1)+K(\phi,-2k) - K(\phi,-2k+1/4) \\
    K_4(\phi,k) & = K(\phi,2k-5/4)-K(\phi,2k-1)+ K(\phi,2k-2) -K(\phi,2k-7/4),
    \end{align*}
where $K(\phi,k):=(\phi+k \, \pi)^{-1}$.
In summary, 
\begin{align*}
    S_{+} & \geq \sum_{k=1}^\infty \int_\omega h^2(\phi)\left(K_1(\phi,k)+K_2(\phi,k)+K_3(\phi,k)+K_4(\phi,k)\right) \dphi.
\end{align*}

Now we look for an upper bound on $\frac{1}{2}[h^-]_{\frac{1}{2}}^2$ by comparing $S_{-}$ and $S_+$.  We have 
\begin{align*}
    S_- & = \sum_{k=1}^\infty \int_0^\pi \int_0^\pi \delta(\theta,\phi)^2T_{2k}(\theta,\phi) \dtheta \dphi + 
    \int_0^\pi \int_0^\pi \sigma(\theta,\phi)^2T_{2k-1}(\theta,\phi) \dtheta\dphi \\
    & \leq 2 \sum_{k=1}^{\infty} \int_0^\pi\int_0^\pi (h^2(\theta)+h^2(\phi))(T_{2k}(\theta,\phi)+T_{2k-1}(\theta,\phi))\dtheta \dphi \\
    & \leq 2 \sum_{k=1}^{\infty} \int_\omega h^2(\phi)K_5(\phi,k) \dphi,
\end{align*}
where
\begin{align*}
    K_5(\phi,k) & = K(\phi,-2k-1)-K(\phi,-(2k-1))+ K(\phi,2k-2) - K(\phi,2k). 
\end{align*}
This follows by repeatedly using the facts that $\int_0^\pi T_m(\theta,\phi) \dtheta=K(\phi,m+1)-K(\phi,m)$ and $\int_0^\pi T_m(\theta,\phi) \dphi= K(\theta,-m-1)-K(\theta,-m)$ for suitable choices of integer $m$.
We claim there is $l>0$ independent of $k$ and $\phi$ such that 
\begin{align}\label{korngold1}
K_1(\phi,k)+K_2(\phi,k)+K_3(\phi,k)+K_4(\phi,k) \geq l \, K_5(\phi,k)
\end{align}
for $k\in \N$ and $\phi \in \omega$.   To prove the claim, a direct calculation using a computer algebra package 
(Maple$^{\textrm{TM}}$ in this case) 
shows that 
\begin{align}
    K_5(\phi,k) & = \frac{1}{\pi k^2} + O\left(\frac{1}{k^3}\right) \quad \textrm{as} \ k \to \infty \label{MathematicaExpr1}\\
      \sum_{j=1}^4 K_j(\phi,k)  & \simeq \frac{1}{2 \pi k^2} + O\left(\frac{1}{k^3}\right) \quad \textrm{as} \ k \to \infty. 
      \label{MathematicaExpr2} \\
 \frac{\sum_{j=1}^4 K_j(\phi,k)}{K_5(\phi,k)}  & \simeq \frac{1}{2} + O\left(\frac{1}{k^2}\right) \quad \textrm{as} \ k \to \infty. 
 \label{MathematicaExpr3} 
\end{align}
In fact, \eqref{MathematicaExpr1} is easily verified by hand, while \eqref{MathematicaExpr2} can be deduced by observing that the behaviour of $\sum_{j=1}^4 K_j(\phi,k)$ for large $k$ is independent of $\phi$, so we formally set $\phi=0$, substitute $y=1/k$ and Taylor expand, in $y$, the expression $K(0,1/y)$.  Therefore, provided $l_1 < \frac{1}{2}$, it follows that
\begin{align}\label{korngold2}
K_1(\phi,k)+K_2(\phi,k)+K_3(\phi,k)+K_4(\phi,k) \geq l_1 \, K_5(\phi,k)
\end{align}
for sufficiently large $k$.   To finish the argument, note that for each fixed $k$ the function $\sum_{j=1}^4K_j(\phi,k)$ is continuous in $\phi \in \omega$ and bounded away from zero, as is each function $K_5(\phi,k)$.  Thus, given $k_0$, there is $l_0>0$ such that \eqref{korngold1} holds for $k=1,\ldots,k_0$ uniformly in $\phi$.  By taking $k_0$ sufficiently large and applying \eqref{korngold2}, we deduce that \eqref{korngold1} holds for all $k \in \N$ and all $\phi \in \omega$ with $l=\min\{l_0,l_1\}$, proving the claim.

To finish the proof of the proposition, we combine the estimates above as follows:
\begin{align*}
  \frac{1}{2}[h^+]_{\frac{1}{2}}^2 & = \int_0^\pi \int_0^\pi  \delta^2(\theta,\phi)T_0(\theta,\varphi) \dtheta\dphi + 2 S_+ \\ 
  & \geq \int_0^\pi \int_0^\pi  \delta^2(\theta,\phi)T_0(\theta,\varphi) \dtheta\dphi + 2 \sum_{k=1}^{\infty} \int_{\omega} h^2(\phi) \sum_{j=1}^4 K_j(\phi,k) \dphi \\
  & \geq \int_0^\pi \int_0^\pi  \delta^2(\theta,\phi)T_0(\theta,\varphi) \dtheta\dphi + 2l \sum_{k=1}^{\infty} \int_{\omega} h^2(\phi) K_6(\phi,k) \dphi \\
  & \geq \int_0^\pi \int_0^\pi  \delta^2(\theta,\phi)T_0(\theta,\varphi) \dtheta\dphi + l S_- \\
  & \geq \frac{L}{2}\left( \int_0^\pi \int_0^\pi  \delta^2(\theta,\phi)T_0(\theta,\varphi) \dtheta\dphi + 2 S_-\right) \\
  & = \frac{l}{4}[h^-]_{\frac{1}{2}}^2.
\end{align*}
Thus, \eqref{sonofhitch} holds with $\gamma_0=\frac{l}{2}$.
\end{proof}

\referee 
\begin{remark}\emph{
In fact, numerical results indicate that the constant $\gamma_0$ in Proposition \ref{comphpm} can be taken to be $\frac{1}{4}$, which amounts to checking that $l=\frac{1}{2}$ is permissible in inequality \eqref{korngold1}. 
We find that the quotient 
\begin{align*}
Q(\phi,k):=\frac{\sum_{j=1}^4 K_j(\phi,k)}{K_5(\phi,k)}
\end{align*}
obeys $Q(\phi,k)\geq \frac{1}{2}$ uniformly in $\phi \in \omega$ for all values of $k$ tested.  We also observed that for fixed $k$, $\phi \mapsto Q(\phi,k)$ appears to be convex in $\phi$ and symmetric about $\phi=\frac{\pi}{2}$, so one approach to establishing the lower bound for all $k$ would be to show that $Q(\frac{\pi}{2},k) \geq \frac{1}{2}$ for all $k$.}
\end{remark}
\EEE

Let us return to the lower bound \eqref{AcesUp}, which we reprint here for convenience:
\begin{align}\label{AcesUpUp}
 \frac{1}{2}I(\varphi,\INS{-c}{0}{ 0}{c} \,) \geq \int_{S_1} |\nabla F_1|^2 \dx + \left(1-\frac{c^2}{4}\right) \int_{S_2} |\nabla F_2|^2 \dx.
\end{align}
Given $f_1$, $F_2$ is completely determined by \referee$\Delta F_2=0$ \EEE and the boundary conditions
\begin{displaymath}\label{F2bc}  F_2(x_1,x_2)=\left\{\begin{array}{l l} f_1(\frac{1}{2},x_2) & \mathrm{if } \ x_1=\frac{1}{2}, \ |x_2| \leq \frac{1}{2} \\
0 & \mathrm{if } \ \frac{1}{2} \leq x_1 \leq 1 \ \mathrm{and }  \ x_2 = \pm \frac{1}{2} \ \mathrm{or } \ x_1 =1  \ \mathrm{and }\ |x_2| \leq \frac{1}{2}.
\end{array}\right.
\end{displaymath}
i.e. $F_2\arrowvert_{\scriptscriptstyle{\Gamma_1}} = f_1$ and $F_2\arrowvert_{\partial R_2 \setminus \scriptscriptstyle{\Gamma_1}} =0.$   In the case of $F_1$ we have 
\begin{align}\label{f1star} \int_{S_1} |\nabla F_1|^2 \dx  \geq \int_{S_1} |\nabla F_1^*|^2 \dx,
\end{align}
where $F_1^*$ minimizes $D(\cdot, S_1)$ subject to the boundary conditions 
\begin{align*}
F_1^*(x)=\left\{\begin{array}{l l} 0 & \textrm{if } x \in \partial S_1, \ x_2 =\pm 1, \\
f_1  & \textrm{if } x \in \Gamma_1.\end{array}\right.
\end{align*}
The natural boundary condition $\partial_1 F_1^*$ then holds on $\Gamma_0$.   The connection with problems $\oplus$ and $\ominus$ can now be made via a suitable conformal transformation, as follows.

First, we consider the problem of minimizing $\D(F_2,R_2)$ subject to the boundary conditions \eqref{F2bc} by recasting in the form of problem $\ominus$.  To begin with, since $F_2\arrowvert_{\scriptscriptstyle{\Gamma_2}}=0$, we can extend $F_2$ to a map on the square $S_2:=[\frac{1}{2},\frac{3}{2}] \times [-\frac{1}{2},\frac{1}{2}]$ by setting $F_2(x_1,x_2)=-F_2(2-x_1,x_2)$ for $1\leq x_2 \leq \frac{3}{2}$.  Then, in particular, $F_2$ will minimize $\D(\cdot,S_2)$ and, by a trivial coordinate translation, $F_2$ is equivalent to a map $F_2'$, say, that minimizes $\D(\cdot,[0,1]^2)$ when subject to the conditions \begin{align} \label{bc:F2primereal}
F_2'(z)=\left\{\begin{array}{l l l} f_1(x_2+\frac{1}{2}) & \mathrm{if } & z=x_2 i \in [0,i] \\ 
0 &  \mathrm{if } & z \in [0,1] \cup [i,1+i] \\
 -f_1(x_2+\frac{1}{2}) & \mathrm{if } & z=1+x_2i \in [1,1+i]. 
  \end{array}\right.
\end{align}
It may help to look at Figure \ref{Byrd1} at this point. 

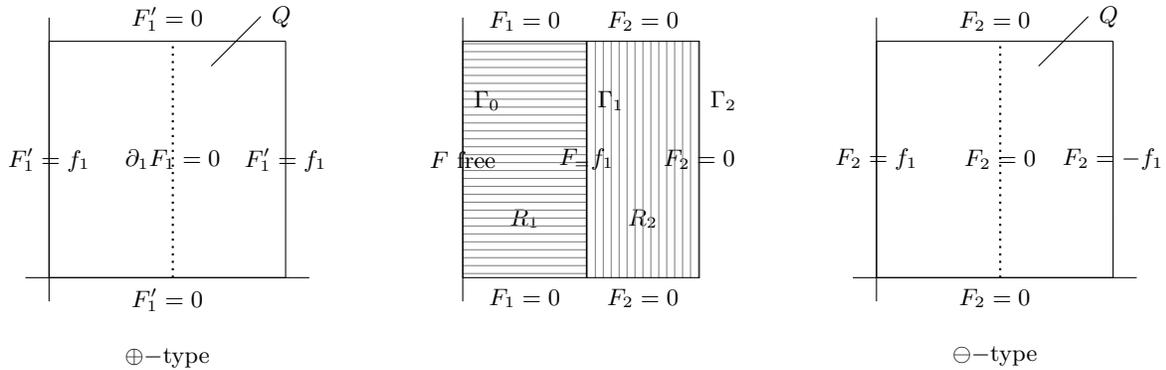
\begin{figure}[ht]
\centering
\begin{tikzpicture}[scale=0.55]

\begin{scope}[shift={(0,0)}]
\node (A) at (0,0) {}; 
\node[right=1.5 of A.center] (B) {};
\node[right=3 of A.center] (C) {};
\node[above=3 of A.center] (A') {};
\node[above=3 of B.center] (B') {};
\node[above=3 of C.center] (C') {};

\draw (A.center) -- (C.center);
\draw (A'.center) -- (C'.center);
\draw (A.center) -- (A'.center);
\draw (C.center) -- (C'.center);
\draw[thick, dotted] (B.center) -- (B'.center);
\draw ($(C')+(-1.8,-0.6)$)-- +(1.2,1.2)  node[right]{\scriptsize $Q$};
\node (U) at ($(A)!0.5!(A')$) {\scriptsize $F_1'=f_1$};
\node (U) at ($(C)!0.5!(C')$) {\scriptsize $F_1'=f_1$};
\node[above] (U) at ($(A')!0.5!(C')$) {\scriptsize $F_1'=0$};
\node[below] (D) at ($(A)!0.5!(C)$) {\scriptsize $F_1'=0$};
\node (U) at ($(B)!0.5!(B')$) {\scriptsize $\partial_1 F_1=0$};
\node[below=0.5 of D.center] (U) {\scriptsize $\oplus-$type};

\node (U) at ($(A)!-0.1!(A')$) {};
\node (V) at ($(A)!1.1!(A')$) {};
\draw (U.center) -- (V.center);
\node (U) at ($(A)!-0.1!(C)$) {};
\node (V) at ($(A)!1.1!(C)$) {};
\draw (U.center) -- (V.center);

\end{scope}
\begin{scope}[shift={(10,0)}]
\usetikzlibrary {patterns,patterns.meta} 
\node (A) at (0,0) {}; 
\node[right=1.5 of A.center] (B) {};
\node[right=3 of A.center] (C) {};
\node[above=3 of A.center] (A') {};
\node[above=3 of B.center] (B') {};
\node[above=3 of C.center] (C') {};

\node (M) at ($(B)!0.5!(B')$) {};
\node (U) at ($(A)!0.5!(M)$) {\scriptsize $R_1$};
\node (U) at ($(C)!0.5!(M)$) {\scriptsize $R_2$};

\filldraw[pattern={horizontal lines},pattern color=gray] (A.center) rectangle (B'.center);
\filldraw[pattern={vertical lines},pattern color=gray] (B.center) rectangle (C'.center);

\node[below] (U) at ($(A)!0.5!(B)$) {\scriptsize $F_1=0$};
\node[below] (U) at ($(B)!0.5!(C)$) {\scriptsize $F_2=0$};
\node[above] (U) at ($(A')!0.5!(B')$) {\scriptsize $F_1=0$};
\node[above] (U) at ($(B')!0.5!(C')$) {\scriptsize $F_2=0$};

\node (U1) at ($(A)!0.5!(A')$) {\scriptsize $F$ free};
\node (U2) at ($(C)!0.5!(C')$) {\scriptsize $F_2=0$};
\node (U) at (M) {\scriptsize $F_=f_1$};
\node[right] (U) at ($(A)!0.75!(A')$) {\scriptsize $\Gamma_0$};
\node[right] (U) at ($(B)!0.75!(B')$) {\scriptsize $\Gamma_1$};
\node[right] (U) at ($(C)!0.75!(C')$) {\scriptsize $\Gamma_2$};

\node (U) at ($(A)!-0.1!(A')$) {};
\node (V) at ($(A)!1.1!(A')$) {};
\draw (U.center) -- (V.center);
\node (U) at ($(U1)!-0.1!(U2)$) {};
\node (V) at ($(U1)!1.1!(U2)$) {};

\end{scope}

\begin{scope}[shift={(20,0)}]
\node (A) at (0,0) {}; 
\node[right=1.5 of A.center] (B) {};
\node[right=3 of A.center] (C) {};
\node[above=3 of A.center] (A') {};
\node[above=3 of B.center] (B') {};
\node[above=3 of C.center] (C') {};

\draw (A.center) -- (C.center);
\draw (A'.center) -- (C'.center);
\draw (A.center) -- (A'.center);
\draw (C.center) -- (C'.center);
\draw[thick, dotted] (B.center) -- (B'.center);
\draw ($(C')+(-1.8,-0.6)$)-- +(1.2,1.2)  node[right]{\scriptsize $Q$};

\node (U) at ($(A)!0.5!(A')$) {\scriptsize $F_2=f_1$};
\node (U) at ($(C)!0.5!(C')$) {\scriptsize $F_2=-f_1$};
\node[above] (U) at ($(A')!0.5!(C')$) {\scriptsize $F_2=0$};
\node[below] (D) at ($(A)!0.5!(C)$) {\scriptsize $F_2=0$};
\node (U) at ($(B)!0.5!(B')$) {\scriptsize $F_2=0$};
\node[below=0.5 of D.center] (U) {\scriptsize $\ominus-$type};

\node (U) at ($(A)!-0.1!(A')$) {};
\node (V) at ($(A)!1.1!(A')$) {};
\draw (U.center) -- (V.center);
\node (U) at ($(A)!-0.1!(C)$) {};
\node (V) at ($(A)!1.1!(C)$) {};
\draw (U.center) -- (V.center);

\end{scope}
\end{tikzpicture}

\caption{This explains the conversion of the minimization problem posed on the central domain $R_1 \cup R_2$ into two separate problems defined on the unit square $Q$.  The minimizer of $\D(F_2,R_2)$ corresponds to the `left-hand half' of the minimizer of $\D(\cdot,Q)$ under the boundary conditions stated on $Q$, suitably shifted, as seen in the problem of $\ominus-$type.  The minimizer of $\D(F_1,R_1)$ corresponds to the `right-hand half' of the minimizer of $\D(\cdot,Q)$ under the boundary conditions stated on $Q$, suitably shifted, as seen in the problem of $\oplus-$type.}\label{Byrd1}
\end{figure}

Let $Q$ be the square in $\C$ with vertices at $0,1,i$ and $1+i$.  
By Proposition \ref{confexist} there exists a conformal map $\psi: D \to Q$, and using it we define an $\R^2$-valued boundary condition on $\partial D$ by setting
\begin{align}\label{bc:h2} h'(\zeta):=F_2'(\psi(\zeta)) \quad \quad \quad \zeta \in \partial D.\end{align}
By Proposition \ref{confexist}, the components $h'_1$ and $h'_2$ of $h'$ are such that each function 
$$ h_{j}(\theta):=h'_j(e^{i\theta})  \quad \quad \quad 0 \leq \theta \leq \pi, \quad  j = 1,2, $$
satisfies assumptions (H1) and (H2).  For $j=1,2$ let $Z_j$ solve problem $\ominus$ with the boundary condition $Z_j\arrowvert_{\partial D}=h_j^{-}$.  Notice that the function $Z'(\xi):=F'_2(\psi(\xi))$, $\xi \in D$, is harmonic and that it obeys the same boundary conditions as $Z$. Hence, by uniqueness, it must be that $Z=Z'$ in $D$, and so by conformal invariance
$$ \D(F_2'^{j},Q)=\D(Z_j,D) \quad \quad j=1,2.$$
Hence, 
\begin{align*}2 \D(F_2,S_2) & = \D(F_2',[0,1]^2) \\
& = \D(F_2'^1,Q)+\D(F_2'^2,Q) \\
& = \D(Z_1,D) + \D(Z_2,D) \\
& = \frac{1}{\pi}[h_1^{-}]^2_{\frac{1}{2}} + \frac{1}{\pi}[h_2^{-}]^2_{\frac{1}{2}}, 
\end{align*} 
where in passing from the third to the fourth line we have used \eqref{faurem}.
In summary, 
\begin{align} \label{ens2} \D(F_2,S_2)=\frac{1}{2\pi}[h_1^{-}]^2_{\frac{1}{2}} + \frac{1}{2\pi}[h_2^{-}]^2_{\frac{1}{2}}, 
\end{align}
A similar argument for $F_1^*$ yields 
\begin{align} \label{ens3} \D(F_1^*,S_1)=\frac{1}{2\pi}[h_1^{+}]^2_{\frac{1}{2}} + \frac{1}{2\pi}[h_2^{+}]^2_{\frac{1}{2}}, 
\end{align}
Proposition \eqref{comphpm} then implies that $\D(F_1^*,S_1) \geq \gamma_0 \D(F_2,S_2)$, and 
hence, from \eqref{AcesUpUp} and \eqref{f1star} it follows that
\begin{align*}
 \frac{1}{2}I(\varphi,\INS{-c}{0}{ 0}{c} \,)  & \geq \D(F_1^*,S_1) +  \left(1-\frac{c^2}{4}\right) \D(F_2,S_2) \\
 & \geq  \left(\gamma_0 + 1-\frac{c^2}{4}\right) \D(F_2,S_2),
\end{align*}
which is nonnegative provided 
$$c \leq 2\sqrt{1 + \gamma_0}.$$
In particular, a total jump in $f$ of $4\sqrt{1 + \gamma_0}$ is compatible with the nonnegativity of $I(\varphi,\INS{-c}{0}{0}{c} \,)$.  The foregoing results are summarised as follows.

\begin{proposition}\label{sum:4+} There exists $\gamma_0>0$ depending only on the domain $\om=R_{-2} \cup R_{-1} \cup R_{1} \cup R_2$ such that the functional $$ I(\varphi,\INS{-c}{0}{0}{c} \,):= \int_{\om} |\nabla \varphi|^2+f(x) \det \nabla \varphi(x) \dx $$
introduced in \eqref{I_four_square} is nonnegative on $W^{1,2}_0(\om,\R^2)$ provided 
$$|c| \leq 2\sqrt{1 + \gamma_0}.$$
\referee Numerical results indicate that $\gamma_0=\frac{1}{4}$, and hence that 
$|c| \leq \sqrt{5}$ is consistent with $$I(\varphi,\INS{-c}{0}{0}{c} \,) \geq 0$$ for all $\varphi$ in $W^{1,2}_0(\om,\R^2)$. \EEE
\end{proposition}

\referee
\begin{remark}\emph{We do not know whether $|c|=\sqrt{5}$ is the optimal, i.e. largest possible, constant for which $I(\varphi,\INS{-c}{0}{0}{c} \,) \geq 0$ for all relevant $\varphi$ in $W^{1,2}_0(\om,\R^2)$. The numerical results of Section \ref{numericalresults} suggest that  $c=4$ is optimal:  see Figure \ref{fig:insulation_intro1}(D).}
\end{remark}
\EEE 
The following result establishes that if $\delta f > 4$ then the width of the `insulating' layer cannot be made as small as please.

\begin{proposition}\label{thinmiddle} 
Let $\Omega$ be a domain of the form 
$$\Omega:=L \cup ([0,\delta]\times [0,1]) \cup R,$$
where $L:=[-1,0] \times [0,1]$, $R:=[\delta,\delta+1]\times [0,1]$ and $\delta >0$, let $c>2$, and let 
\begin{align}\label{fff}f=c\chi_L - c\chi_R.\end{align}
Then 
there is a minimum `width' $\delta$ for which the inequality
\begin{align*}
    \int_{\Omega} |\nabla \varphi|^2 + f(x)\det \nabla \varphi \, \dx \geq 0 \quad \quad \varphi \in W_0^{1,2}(\Omega;\R^2)
\end{align*}
holds.
\end{proposition}
\begin{proof} 
Choose any $c'$ such that $c > c'>2$ and apply Proposition \ref{prop:n_dim_two_state} to deduce that there exists a $\varphi \in W_{\Gamma}^{1,2}(L;\R^2)$ such that 
\begin{align}\label{seeprime} \int_L |\nabla \varphi|^2 + c'\det \nabla \varphi \, \dx < 0,\end{align}
 where $\Gamma:=\{0\} \times [0,1]$.
Without loss, we may assume that $\varphi$ is smooth on $L$.   Using \eqref{seeprime}, a short calculation shows that 
\begin{align}
   \int_L |\nabla \varphi|^2 + c\det \nabla \varphi \, \dx  < -\frac{(c-c')}{c'} \int_{L} |\nabla \varphi|^2 \, \dx.
\end{align}
We now extend $\varphi$ to a map $\Phi$ in $W_0^{1,2}(\Omega;\R^2)$  by setting
\begin{displaymath}
\Phi(x_1,x_2):=\left\{\begin{array}{ll}
    \varphi(x_1,x_2) & \textrm{if} \ x \in L \\[2mm]
     \varphi(0,x_2) & \textrm{if} \ x \in [0,\delta]\times [0,1] 
     \\[2mm]
     \varphi(\delta-x_1,x_2) & \textrm{if} \ x \in R.
\end{array}\right.
\end{displaymath}
It follows that 
\begin{align*}\int_{\Omega} |\nabla \Phi|^2 + f(x) \det \nabla \Phi \, \dx & = 2 \int_L |\nabla \varphi|^2 + c\det \nabla \varphi \dx  + \delta \int_0^1 |\partial_2\varphi(0,x_2)|^2 \, dx_2 \\
& < -\frac{2(c-c')}{c'} \int_{L} |\nabla \varphi|^2 \, \dx + \delta \int_0^1 |\partial_2\varphi(0,x_2)|^2 \, dx_2,
\end{align*}
the right-hand side of which can be made negative by taking $\delta$, which in this setting is the width of the insulation region $\{f=0\}$, sufficiently small.  We conclude that for any $f$ of the form \eqref{fff} with $c>2$, the insulation region cannot be made arbitrarily small.
    \end{proof}

\subsection{The point-contact problem}\label{-4040}

The geometry of the insulation region that featured in the previous example kept the sets $\{f=\pm c\}$ completely separate, and so produced an example of (HIM) in which $\delta f > 4$.  The following example, which explores a different insulation geometry, shows that (HIM) can also be made to hold when the sets $\{f=\pm c\}$ meet in a point and $\delta f=2c>4$.
In fact, we obtain (HIM) when $\delta f = 2\sqrt{8}$, and note that, owing to the zero-homogeneity of the $f$ we consider, (HIM) is equivalent to the sequential weak lower semicontinuity of the associated functional. 

Let $Q=[-1,1]^2$ and let $Q_1,\ldots,Q_4$ be the four `windows' of $Q$ shown in Figure \ref{pic:distribution_Q}.


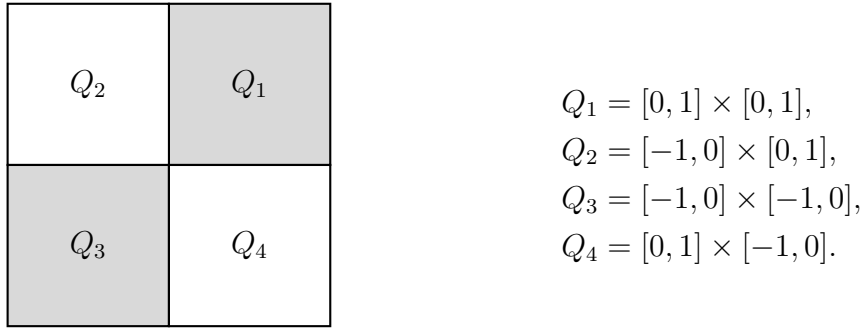
\begin{figure}[H]
\centering
\begin{minipage}{0.39\textwidth}
\centering
\begin{tikzpicture}[scale=2]
\node (A) at (-2,-2) {}; 
\node[right=2 of A.center] (B) {};
\node[right=2 of B.center] (C) {};
\node[right=2 of C.center] (D) {};
\node[above=2 of A.center] (A') {};
\node[above=2 of B.center] (B') {};
\node[above=2 of C.center] (C') {};
\node[above=2 of A'.center] (A'') {};
\node[above=2 of B'.center] (B'') {};
\node[above=2 of C'.center] (C'') {};

\filldraw[thick, top color=gray!30!,bottom color=gray!30!] (B'.center) rectangle node{$Q_1$} (C''.center);

\filldraw[thick, top color=white,bottom color=white] (A'.center) rectangle node{$Q_2$} (B''.center);

\filldraw[thick, top color=gray!30!,bottom color=gray!30!] (A.center) rectangle node{$Q_3$} (B'.center);

\filldraw[thick, top color=white,bottom color=white] (B.center) rectangle node{$Q_4$} (C'.center);

\end{tikzpicture}
\end{minipage}
\begin{minipage}{0.39\textwidth}
\begin{align*}
    Q_1&= [0,1] \times [0,1], \\
    Q_2 & = [-1,0] \times [0,1], \\
    Q_3 & = [-1,0] \times [-1,0], \\
    Q_4 & = [0,1] \times [-1,0].
\end{align*}
\end{minipage}
\caption{Distribution of $Q_1,\ldots,Q_4$.} \label{pic:distribution_Q}
\end{figure}

Denote by 
\begin{align}\label{boxedup}I\left(\varphi, 
\scalebox{0.7}{{\begin{tabular}{|c|c|}
 \hline
$a_2$ &$a_1$ 
\\ \hline
$a_3$ & $a_4$ \\
\hline
\end{tabular}}}\, \right) & =\sum_{i=1}^{4} \int_{Q_i}|\nabla \varphi|^2 + a_i \det \nabla \varphi \dx.
\end{align}

\begin{proposition}\label{foursquares1} Let $Q=[-1,1]^2$.  Then \begin{itemize}
    
    \item[(a)] if $\varphi$ belongs to the closure of the set 
$$
    \mathcal{C}:=\{\varphi \in C_c^{\infty}(Q;\R^2): \ \varphi(0,x_2)=\varphi(0,-x_2) \ \textrm{for} \ |x_2|\leq 1\}
$$
in the $H^1$ norm, then 
$$ I\left(\varphi, 
\scalebox{0.7}{{\begin{tabular}{|c|c|}
 \hline
$0$ &$-4$ 
\\ \hline
$4$ & $0$ \\
\hline
\end{tabular}}}\, \right) \geq 0,$$
and the result is optimal in the sense that it is false when $4$ is replaced with any $k>4$;
\item[(b)] it holds that
$$
I\left(\varphi,  
\scalebox{0.7}{{\begin{tabular}{|c|c|}
 \hline
$0$ &$-c$ 
\\ \hline
$c$ & $0$ \\
\hline
\end{tabular}}}\, \right)  \geq 0 \quad \quad \quad \forall \varphi \in W^{1,2}_0(Q;\R^2)
$$
iff $|c|\leq \sqrt{8}$.
\end{itemize}

\end{proposition}
\begin{proof}
\noindent \textbf{(a)}  Let $\varphi$ belong to $\mathcal{C}$.  We may assume that $\varphi$ is harmonic in each of $Q_1, \ldots Q_4$.  Write $\varphi= f + g$, where $f$ is the even part of $\varphi$ and $g$ the odd part.  Then $\nabla f$ is odd and $\nabla g$ is even, and in particular
\begin{align*}
    \int_Q \nabla f \cdot \nabla g \dx & = 0 \\
    \int_{Q_3} \cof \nabla f \cdot \nabla g \dx &  = - \int_{Q_1} \cof \nabla f \cdot \nabla g \dx \\
    \int_{Q_3} \det \nabla f \dx & =  \int_{Q_1} \det \nabla f \dx \\
     \int_{Q_3} \det \nabla g \dx & =  \int_{Q_1} \det \nabla g \dx. 
\end{align*}
Hence 
\begin{align}
    I\left(\varphi, 
\scalebox{0.7}{{\begin{tabular}{|c|c|}
 \hline
$0$ &$-4$ 
\\ \hline
$4$ & $0$ \\
\hline
\end{tabular}}}\, \right) & = 2 \int_{Q_1\cup Q_2} |\nabla f|^2+|\nabla g|^2 \dx - 4\int_{Q_1} \det \nabla f + \det \nabla g + \cof \nabla f \cdot \nabla g \dx \nonumber  \\& \quad \quad \quad \quad \quad \quad \quad \quad \quad \quad \quad \quad \quad
+ 4 \int_{Q_3} \det \nabla f + \det \nabla g + \cof \nabla f \cdot \nabla g \dx \nonumber \\
& =  2 \int_{Q_1\cup Q_2} |\nabla f|^2+|\nabla g|^2 \dx - 8 \int_{Q_1} \cof \nabla f \cdot \nabla g \dx. \label{NationalExpress}
\end{align}
Since $f\arrowvert_{Q_1}$ is harmonic and $f(-x_1,0)=f(x_1,0)$ for $|x_1| \leq 1$, it must be that $f\arrowvert_{Q_2}$ agrees with 
$$f^*(x_1, x_2):=f(-x_1,x_2) \quad \quad \quad (x_1,x_2) \in Q_2$$ and in particular 
$$\int_{Q_2} |\nabla f|^2 \dx= \int_{Q_2} |\nabla f^*|^2 \dx = \int_{Q_1} |\nabla f|^2 \dx .$$ A similar argument applies to $g$, but here we set $g^*$ in $Q_2$ to be 
$$g^*(x_1,x_2):=-g(-x_1,x_2) \quad \quad \quad (x_1,x_2) \in Q_2.$$ We see that, because $g$ is odd, $g^*(x_1,0)=g(x_1,0)$ and, owing to the requirement that $\varphi(0,x_2)=\varphi(0,-x_2)$,  $g^*(0,x_2)=g(0,x_2)=0$. It follows that $g^*$ and $g$ must agree on $Q_2$, and hence that 
$$\int_{Q_2} |\nabla g|^2 \dx =\int_{Q_2} |\nabla g^*|^2 \dx =\int_{Q_1} |\nabla g|^2 \dx.$$
Thus
\begin{align}
    I\left(\varphi, 
\scalebox{0.7}{{\begin{tabular}{|c|c|}
 \hline
$0$ &$-4$ 
\\ \hline
$4$ & $0$ \\
\hline
\end{tabular}}}\, \right) & =4 \int_{Q_1} |\nabla f|^2+|\nabla g|^2 \dx - 2\, \cof \nabla f \cdot \nabla g \dx \nonumber \\
& = 4 \int_{Q_1} |\nabla g - \cof \nabla f|^2 \dx,
\end{align}
which is always non-negative.  \referee Now suppose that the entries $\pm 4$ are replaced by $\pm k$ for some $k>4$.  Fix a radius $r<\frac{1}{4}$, let $x_0=(\frac{1}{2},0)$ and note that the ball $\overline{B(x_0,r)}$ is strictly contained in $Q_1 \cup Q_4$. Applying Proposition \ref{islandproblem} with $n=2$, choose  $\varphi \in W_0^{1,2}(B(x_0,r),\R^2)$ such that 
\begin{align*}
    \int_{Q_1 \cup Q_4}|\nabla \varphi|^2 - k \det \nabla \varphi \, \chi_{Q_1}\dx < 0.\end{align*}
Extend $\varphi$ to $\tilde{\varphi} \in W_0^{1,2}(Q,\R^2)$ by setting $\tilde{\varphi}=0$ on $Q \setminus \overline{B(x_0,r)}$. Then by suitably mollifying $\tilde{\varphi}$, but without relabelling, we obtain $\tilde{\varphi} \in \mathcal{C}$ and 
\begin{align*}
   I\left(\tilde{\varphi}, 
\scalebox{0.7}{{\begin{tabular}{|c|c|}
 \hline
$0$ &$-k$ 
\\ \hline
$k$ & $0$ \\
\hline
\end{tabular}}}\, \right) & < 0. 
\end{align*}
Hence the constant $k=4$ is optimal in this case. \EEE

\noindent \textbf{(b) } We prove first that $|c| \leq \sqrt{8}$ is sufficient.  By replacing $\varphi$ with $J\varphi$ if necessary, we may assume that $c\geq 0$. Arguing to start with as in part (a), we find that
\begin{align*}  \small{I\left(\varphi,  
\scalebox{0.7}{{\begin{tabular}{|c|c|}
 \hline
$0$ &$-c$ 
\\ \hline
$c$ & $0$ \\
\hline
\end{tabular}}}\, \right)} & = 2\int_{Q_1 \cup Q_2} |\nabla g|^2 \dx +  4\int_{Q_1} |\nabla f|^2 \dx -2c\int_{Q_1}  \cof \nabla f \cdot \nabla g \dx \\
& \geq \int_{Q_1} 2 |\nabla g|^2 + 4 |\nabla f|^2 + c\left|\frac{\nabla f}{\lambda} - \lambda \, \cof \nabla g \right|^2 -c \left|\frac{\nabla f}{\lambda}\right|^2 - 
c \left|\lambda \, \cof \nabla g\right|^2
\dx \\
& \geq \int_{Q_1} 2 |\nabla g|^2 + 4 |\nabla f|^2 -c \left|\frac{\nabla f}{\lambda}\right|^2 - 
c \left|\lambda \, \cof \nabla g\right|^2
\dx,
\end{align*}
with equality (in the last line) iff $\nabla f = \lambda^2 \, \cof \nabla g$.  Setting $\rho=\lambda^2$ and replacing terms in $\nabla f$ with suitable terms in $\nabla g$, we obtain the lower bound
\begin{align*}  \small{I\left(\varphi,  
\scalebox{0.7}{{\begin{tabular}{|c|c|}
 \hline
$0$ &$-c$ 
\\ \hline
$c$ & $0$ \\
\hline
\end{tabular}}}\, \right)} & \geq  \int_{Q_1} (4\rho^2 + 2 - 2c\rho)|\nabla g|^2 \dx, 
\end{align*}
which is nonnegative for all $\rho$ if $c \leq \sqrt{8}$.

The necessity of $|c|\leq \sqrt{8}$ follows by applying Lemmas \ref{personanongrata} and \ref{veryfast} below.  
\end{proof}

\begin{lemma}\label{personanongrata} 
Let $Q_1$ be the unit square in $\R^2$, let $N$ be a large natural number, and let $A$, $B$ be harmonic functions on $Q_1$ obeying the following boundary conditions 
\begin{displaymath}
\begin{array}{l l l}
A(1,x_2)  = 0 & \quad B(1,x_2)=0 & \quad 0 \leq x_2 \leq 1 \\
A(x_1,1)  = 0 & \quad B(x_1,1)=0 & \quad 0 \leq x_1 \leq 1 \\
A(0,x_2) = 0  & \quad B(0,x_2)=g_0(0,x_2) & \quad 0 \leq x_2 \leq 1 \\
A(x_1,0)  = g_0(x_1,0) & \quad B(x_1,0)=0 & \quad 0 \leq x_1 \leq 1,
\end{array}
\end{displaymath}
where the function $g_0$ is such that
\begin{align*}
g_0(x_1,0) & =\sum_{n=1}^N -\frac{1}{n} \sin(n \pi x_1) \quad \quad 0 \leq x_1 \leq 1 \\
g_0(0,x_2) & = \sum_{n=1}^N \frac{1}{n} \sin(n \pi x_2) \quad \quad 0 \leq x_2 \leq 1.
\end{align*}
Let 
\begin{align*} g(x)=\left\{\begin{array}{l l} (B(x)+A(x))e_1 & \textrm{if} \ x \in Q_1 \\
(B(R x)-A(R x))e_1 & \textrm{if} \ x \in Q_2,  \end{array} \right.
\end{align*}
where $R(x_1,x_2)=(-x_1,x_2)$. 
Then, as $N \to \infty$, 
\begin{align}\label{logN}
\mathbb{D}(g;Q_1) & = 2\pi \ln N+O(1) \\
\label{notlogN} \mathbb{D}(g;Q_2) & = O(1).
\end{align}
\end{lemma}
\begin{proof} By a direct calculation, one finds that 
\begin{align}\label{a}
    A(x_1,x_2) & = \sum_{k=n\pi,n \in \N}\alpha_k \sin(kx_1)\sinh(k(x_2-1)), 
    \end{align}
where 
\begin{align*}
    \alpha_k & = \frac{2\int_0^1 g_0(s,0)\sin(ks) \, ds}{\sinh(-k)}=\left\{\begin{array}{l l} \frac{1}{n\sinh(n\pi)} & \textrm{if} \ 1 \leq n \leq N  \\
0 & \textrm{otherwise}.\end{array}\right.\end{align*}
Hence,
\begin{align*}
A(x_1,x_2) & = \sum_{n=1}^N \frac{ \sin(n \pi x_1)\sinh(n \pi (x_2-1))}{n\sinh(n\pi)}. 
\end{align*} 
Similarly, 
\begin{align}\label{b}
    B (x_1,x_2) & =\sum_{n=1}^N \frac{ -\sin(n \pi x_2)\sinh(n \pi (x_1-1))}{n\sinh(n\pi)}.\end{align}
(Note that $B(x_1,x_2)=-A(x_2,x_1)$.) Using these expressions together with the fact that $A$ and $B$ are harmonic, we find that 
 \begin{align}\nonumber \int_{Q_1} \nabla A \cdot \nabla B \dx  & = \sum_{n,m=1}^N\frac{-nm}{n^2+m^2}\alpha_{n\pi}\beta_{m\pi}\sinh(m \pi)\sinh(n\pi), \\
\nonumber & = \sum_{n,m=1}^N \frac{1}{n^2+m^2} \\
\nonumber & = \sum_{n=1}^N \frac{\pi n \coth(\pi n) - 1}{2n^2} + O(1) \\
\nonumber & = \sum_{n=1}^N \frac{\pi\coth(\pi n)}{2n} + O(1) \\
\label{iplog}& = \frac{\pi \ln N}{2} + O(1)  \quad \textrm{as} \ N \to \infty.
\end{align}
In the course of the above, the fact that 
\begin{align*}
    \sum_{m =1}^N \frac{1}{m^2+n^2} = \frac{\pi n \coth(\pi n) - 1}{2 n^2} + O(1/N),
\end{align*}
which can be deduced by integrating the function $h(z):=\frac{\cot(\pi z)}{z^2+n^2}$ around a rectangular contour with corners at $\pm(N+\frac{1}{2})\pm 2Ni$, has been used. 
Similarly, \begin{align}\nonumber
\mathbb{D}(B;Q_1)=\mathbb{D}(A;Q_1) & = \sum_{n=1}^N  \frac{n\pi \, \alpha_{n\pi}^2\sinh(2n\pi)}{4} \\
\nonumber & = \sum_{n=1}^N \pi \frac{n \sinh(2n\pi)}{4 n^2 \sinh^2(n\pi)} \\
\nonumber & = \sum_{n=1}^N \frac{ \pi \coth(\pi n)}{2n} \\
\label{dirABlog} & = \frac{\pi \ln N}{2} +O(1)  \quad \textrm{as} \ N \to \infty.
\end{align}

To conclude the proof, we see from the fact that $g=A+B$ in $Q_1$, \eqref{iplog} and \eqref{dirABlog} that, as $N \to \infty$,
\begin{align*} \mathbb{D}(g;Q_1)=2\pi \ln N + O(1)
\end{align*} 
On $Q_2$, $g(x)=B(Rx)-A(Rx)$, and hence 
\begin{align}\label{dirgq2} \mathbb{D}(g;Q_2)=2\left(\mathbb{D}(A;Q_1)-\int_{Q_2} \nabla \eta \cdot \nabla \zeta \dx\right),
\end{align} 
where $\eta(x):=B(Rx)$ and $\zeta(x):=A(Rx)$.  We calculate  that 
\begin{align*} \int_{Q_2} \nabla \eta \cdot \nabla \zeta \dx & = \sum_{n,m=1}^N \frac{1}{n^2+m^2}, 
\end{align*}
which, by the argument leading to \eqref{iplog}, goes as $\frac{\pi \ln N}{2} +O(1)$  as $N \to \infty.$ Inserting this and \eqref{dirABlog} into \eqref{dirgq2} leads to the conclusion that $\mathbb{D}(g;Q_2)=O(1)$ as $N \to \infty$. 
\end{proof}

In the following, for $S \subset \R^2$ we let $S_1$ and $(S)_1$ refer to the part of $S$ that lies in the first quadrant, $S_2$ and $(S)_2$ to the part that lies in the second quadrant, and so on. 

\begin{lemma}\label{veryfast} Let $r>1$, $\rho>0$ and $N \in \N$, let $Q=[-1,1]^2$ and set $rQ= \{rq: q \in Q\}$.  Then there exists $\Phi \in H_0^1(rQ,\R^2)$ such that 
\begin{align} \label{expansion} \small{I\left(\Phi,  \scalebox{0.7}{{\begin{tabular}{|c|c|}
 \hline
$0$ &$-c$ 
\\ \hline
$c$ & $0$ \\
\hline
\end{tabular}}}\, \right)} 
& = 4\pi (1+ 2\rho^2 - c\rho)\ln N + O(1)
\end{align}
as $N \to \infty$.  In particular, $|c|\leq\sqrt{8}$ is necessary for the everywhere positivity of the functional in the left-hand side of \eqref{expansion}.
\end{lemma}

\begin{proof} 
\noindent \textbf{Step 1.}  Let $g: Q_1 \cup Q_2 \to \R^2$ be the function constructed in Lemma \ref{personanongrata}, and let $f: Q_1 \to \R^2$ be given by $f(x)=f_2(x) e_2$, where 
$$f_2(x)=\rho\left(u_{\scriptscriptstyle{N}}(x_1,x_2)+u_{\scriptscriptstyle{N}}(x_2,x_1)-2 \sum_{n=1}^N \frac{(-1)^n}{n \sinh(\pi n)}\right)$$
and
$$u_{\scriptscriptstyle{N}}(x_1,x_2) = \sum_{n=1}^N \frac{\cos(n\pi x_1)\cosh(n\pi(x_2-1))}{n \sinh(\pi n)}.$$
Then \begin{align*}\nabla f = \rho \, \cof \nabla g \quad  \textrm{in} \ Q_1. \end{align*}Define
\begin{align}\label{tildef}
\tilde{f}(x_1,x_2) = \left\{\begin{array}{l l } f(x_1,x_2) &  \textrm{if} \ 0 \leq x_1 \leq 1, \ 0 \leq x_2 \leq 1 \\
f(x_1,1)a(x_2) &  \textrm{if} \ 0 \leq x_1 \leq 1, \ 1 \leq x_2 \leq r \\
f(1,x_2)a(x_1) &  \textrm{if} \ 1 \leq x_1 \leq r, \ 0 \leq x_2 \leq 1 \\
0 & \textrm{otherwise}.
\end{array} \right.
\end{align}
Here, $a(s)=\frac{r-s}{\eps}$ where $r=1+\eps$ defines $\eps>0$.  Then, by a direct calculation,  
\begin{align}\label{greatscott}\mathbb{D}(\tilde{f},(rQ)_1\setminus Q_1) = \frac{2}{\eps}\int_0^1 f_2^2(x_1,1) \dx+ \frac{2\eps}{3}\int_0^1 f^2_{_{2,1}}(x_1,1) \dx_1.
\end{align}
In the following, let $s_n=\sinh(n \pi)$.  Then
\begin{align*} f_{_{2,1}}(x_1,1)=\pi \rho \left(\sum_{n=1}^N \frac{(-1)^n \sinh(\pi n (x_1-1))}{s_n} -\frac{\sin(\pi n x_1)}{s_n}\right).  
\end{align*}
Both sums are bounded independently of $N$.   This is obvious for the second sum, while the first is easily handled by splitting into the cases $0< x_1 \leq 1$ and $x_1=0$.   Since $f_2(1,1)=0$, the Poincar\'{e} inequality 
\begin{align*} \int_0^1 f_2^2(x_1,1) \dx_1 \leq \frac{4}{\pi^2} \int_0^1 f^2_{_{2,1}}(x_1,1) \dx_1
\end{align*}
applies, and hence, by \eqref{greatscott}, $\mathbb{D}(\tilde{f},(rQ)_1\setminus Q_1)$ is bounded independently of $N$.  \\

\vspace{1mm}
\noindent \textbf{Step 2.}  Construct $\Phi$ as follows.  Firstly, extend $g$ by zero into $((rQ)_1 \cup (rQ)_2) \setminus (Q_1 \cup Q_2)$ and call the resulting function $G(X)$, where $X=(x_1,x_2)$. Then set 
\begin{align*} \tilde{G}(X) & = \left\{ \begin{array}{l l } G(X) & \textrm{if } X \in \overline{(rQ)_1 \cup (rQ)_2} \\
-G(-X) & \textrm{if } X \in \overline{(rQ)_3 \cup (rQ)_4} \end{array}\right.
\end{align*}
Next, define $F: (rQ)_1 \cup (rQ)_2 \to \R^2$ by reflecting $\tilde{f}$ in the (upper) $x_2$-axis, i.e.
\begin{align*} F(X) & = \left\{ \begin{array}{l l } \tilde{f}(X) & \textrm{if } X \in \overline{(rQ)_1} \\
\tilde{f}(-x_1,x_2) & \textrm{if } X \in \overline{(rQ)_2}, \end{array}\right.
\end{align*}
and then extend $F$ to a function $\tilde{F}: rQ \to \R^2$ by 
\begin{align*} \tilde{F}(X) & = \left\{ \begin{array}{l l } F(X) & \textrm{if } X \in \overline{(rQ)_1 \cup (rQ)_2} \\
F(-X) & \textrm{if } X \in \overline{(rQ)_3 \cup (rQ)_4}. \end{array}\right.
\end{align*}
Note that by symmetry,
\begin{align*} \int_{(rQ)_3 \cup (rQ)_4} |\nabla \tilde{F}|^2 & = \int_{(rQ)_2 \cup (rQ)_1} |\nabla F|^2 
\end{align*}
and
\begin{align*}
\int_{(rQ)_2}|\nabla F|^2 & = \int_{(rQ)_1} |\nabla F|^2 \\
& = \int_{Q_1} |\nabla f|^2 + \mathbb{D}(\tilde{f},(rQ)_1\setminus Q_1) \\
& = \rho^2 \int_{Q_1} |\nabla g|^2 + \underbrace{\mathbb{D}(\tilde{f},(rQ)_1\setminus Q_1)}_{=:D_0}.\end{align*}
(This is where $\nabla f = \rho \, \cof \nabla g$ is used.)
Also, by \eqref{logN} and \eqref{notlogN},
\begin{align*} \int_{Q_1 \cup Q_2} |\nabla G|^2 & = \int_{Q_1} |\nabla g|^2 + \int_{Q_2} |\nabla g|^2 \\
& = 2\pi \ln N + O(1)
\end{align*}
as $N \to \infty$.   Finally, set
\begin{align*} \Phi(X)=\tilde{F}(X)+\tilde{G}(X) \quad \quad \quad X \in rQ.\end{align*}
Then $F$ and $G$ form the even and odd parts respectively of $\Phi \in H_0^1(rQ,\R^2)$. 

\vspace{1mm}
\noindent \textbf{Step 3.} We know that for odd/even decompositions, 
\begin{align*} \small{I\left(\phi,  
\scalebox{0.7}{{\begin{tabular}{|c|c|}
 \hline
$0$ &$-c$ 
\\ \hline
$c$ & $0$ \\
\hline
\end{tabular}}}\, \right)} & = 2\int_{Q_1 \cup Q_2} |\nabla G|^2 + 2 \int_{(rQ)_2 \cup (rQ)_1} |\nabla F|^2 -2c\int_{Q_1}  \cof \nabla F \cdot \nabla G \\
& = 2 \int_{Q_1}|\nabla g|^2 +2 \int_{Q_2} |\nabla g|^2 + 4  \int_{(rQ)_1} |\nabla F|^2 -2c\int_{Q_1}  \cof \nabla f \cdot \nabla g \\
& = 2 \int_{Q_1}|\nabla g|^2 +  4 \rho^2 \int_{Q_1} |\nabla g|^2 + 4D_0 -2c\rho \int_{Q_1}|\nabla g|^2 + O(1) \\
& = 4\pi \ln N + 8 \pi \rho^2 \ln N - 4c \rho \pi \ln N + O(1) \\
& = 4\pi (1+2\rho^2 - c\rho)\ln N + O(1)
\end{align*}
as $N \to \infty$.  When $c >\sqrt{8}$, the prefactor of $\ln N$ in the line above can be made negative by choosing $\rho$ suitably, so (HIM) fails when the domain of integration is $rQ$.  But since the prefactor of $\det \nabla \Phi$ is a $0-$homogeneous function, (HIM) is independent of dilations of the domain.  In particular, by exhibiting $\Phi$ that breaks (HIM) on some $rQ$ we may conclude that (HIM) fails on all $r'Q$ if $c > \sqrt{8}$.  A similar argument works in the case $c < -\sqrt{8}$.
\end{proof}

\section{Numerical experiments in the planar case}\label{numericalresults}
A minimizer $\varphi = (\varphi_1, \varphi_2) $ of \eqref{eye} for $n=2$ can be obtained numerically by the finite element method. We modified the MATLAB tool \cite{mova} exploiting the lowest-order (known as P1) elements  defined on a regular triangulation of $\Omega$. A complementary code is available at
\vspace{-3mm}
\begin{center}
\url{https://www.mathworks.com/matlabcentral/fileexchange/130564} 
\end{center}
\vspace{0.5mm}
for download and testing. The function $f$ is assumed to be a piecewise constant in smaller subdomains. If the triangulation is aligned with subdomain shapes, then the numerical quadrature of both terms in \eqref{eye} is exact. Based on the initial guess provided, the trust-region method strives to find the minimizer. If an argument is found for which the energy value drops below a negative prescribed value (set  in all experiments as $-100$), the computation ends with the output that the problem is unbounded. Otherwise, the problem is bounded and the minimum energy equals zero. Due to a termination criterion of the trust-region method, the zero value is only indicated by a very small positive 
number (eg. $10^{-6}$).

\begin{figure}[H]
\includegraphics[width=\textwidth]{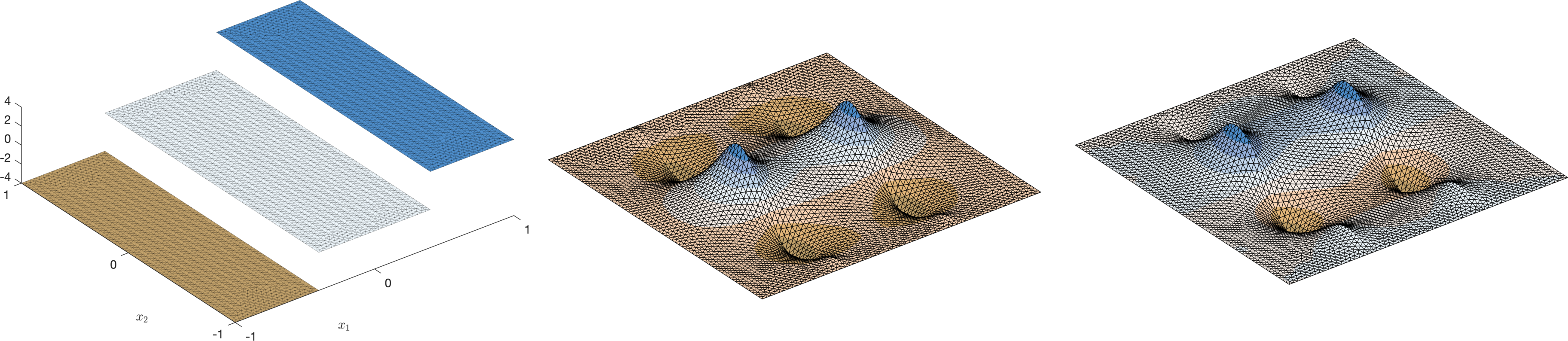}
\caption{Distribution of $f$ (left), components $\varphi_1$ (middle), and $\varphi_2$ (right) of the corresponding minimizer $\varphi=(\varphi_1, \varphi_2)$ providing $I(\varphi)<0$.
}
\label{fig:insultation_intro_minimizers}
\end{figure}

An example of an unbounded insulation type problem is shown in Figure \ref{fig:insultation_intro_minimizers}. It is defined on $\Omega = (-1,1)^2$
with one inner subdomain of vertical boundaries located at $x_1 = \pm 0.4$. Discrete values of $f$ are -4, 0, -4.

A different graphical output is provided in Figures \ref{fig:insulation_intro1}, \ref{fig:insulation_intro2} and \ref{gridmins}, consisting of a 2D view of $f$, low-energy densities 
$$|\nabla \varphi|^2\pm 2 \det \nabla \varphi, \quad |\nabla \varphi|^2 + f \det \nabla \varphi $$ and the deformed triangulated domain $\varphi(\Omega)+\Omega$.  Note that the low-energy $\varphi$
have been rescaled by a constant (positive) factor so that the triangles in the deformed domain provide a smooth visual output.  In view of Proposition \ref{portmanteau}(i), the rescaling preserves the sign of the corresponding energy. 

\subsection{Variable width}
In Figure \ref{fig:insulation_intro1} the deflection angle $\alpha$, 
which is described by means of Figure \ref{pic:alphadescription}, is set to zero and the quantities $\rho:=w_0/w_4$ are varied with $w_{-4}=w_4$. 

\begin{figure}[H]
\centering
\begin{minipage}{1\textwidth}
\centering
\begin{tikzpicture}[scale=0.8]
\node (A) at (-3,-3) {}; 
\node[right=1 of A.center] (B) {};
\node[right=3 of B.center] (C) {};
\node[right=1 of C.center] (D) {};
\node[above=4 of A.center] (A') {};
\node[right=1 of A'.center] (B') {};
\node[right=3 of B'.center] (C') {};
\node[right=1 of C'.center] (D') {};

\node   (B'C')  [above=0.8 of $(B')!0.5!(C')$] {};
\node   (BC)  [below=0.8 of $(B)!0.5!(C)$] {};

\draw [thick] (B.center) -- (BC.center) -- (C.center) -- (C'.center) --   (B'C'.center) -- (B'.center) -- cycle;

\node (M) at ($(BC)!0.5!(B'C')$) {0};    

\filldraw[thick, top color=gray!30!,bottom color=gray!30!] (A.center) rectangle node{$-4$} (B'.center);
\filldraw[thick, top color=gray!30!,bottom color=gray!30!] (C.center) rectangle node{$4$} (D'.center);
\draw [decorate,decoration={brace,amplitude=10pt,mirror,raise=4pt},yshift=-1pt]
(A.center) -- (B.center) node [black,midway,xshift=0cm,yshift=-0.8cm] {$w_{-4}$};
\draw [decorate,decoration={brace,amplitude=10pt,raise=4pt},yshift=-1pt]
(B) -- (C) node [black,midway,xshift=0cm,yshift=0.8cm] { $w_{0}$};
\draw [decorate,decoration={brace,amplitude=10pt,mirror,raise=4pt},yshift=-1pt]
(C.center) -- (D.center) node [black,midway,xshift=0cm,yshift=-0.8cm] {$w_{4}$};

\draw[thick,<->] ($(B'.center)!0.5!(C'.center)$) arc[start angle=0,end angle=30,radius=2]
 node[midway,fill=white] {$\alpha$};

\draw[dotted] (B') -- (C');

\end{tikzpicture}
\end{minipage}
\caption{Insulation problem with deflection.} \label{pic:alphadescription}
\end{figure}
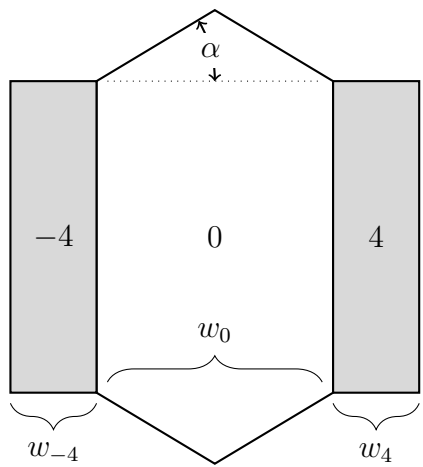

As $\rho$ increases, so the (light blue) insulation region $\{f=0\}$ grows relative to the regions $\{f=\pm 4\}$ and appears to carry with it an increased Dirichlet energy cost (since the integrand becomes $|\nabla \varphi|^2$ when $f=0$) until, in Figure \ref{fig:insulation_intro1} (D), the functional attains a minimum value of zero.\footnote{We conjecture that, in this geometry, (HIM) holds for all $\rho \geq 2$.}  We ask why this should  be so.  

\begin{figure}
\begin{subfigure}[b]{\textwidth}
\centering
\includegraphics[width=\textwidth]{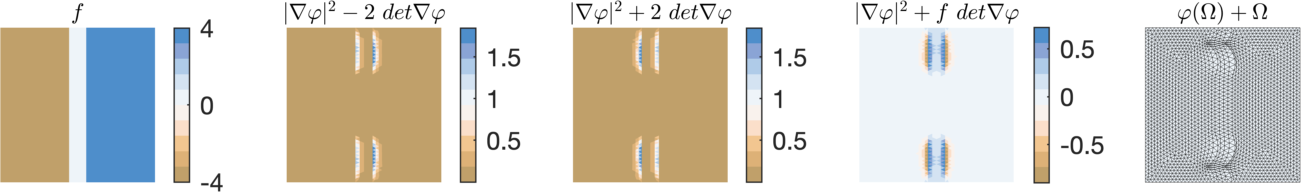}
\caption{$\rho=1/4$:  $\bf{I(\varphi)<0}$}
\vspace{1em}
\end{subfigure}
\begin{subfigure}[b]{\textwidth}
\centering
\includegraphics[width=\textwidth]{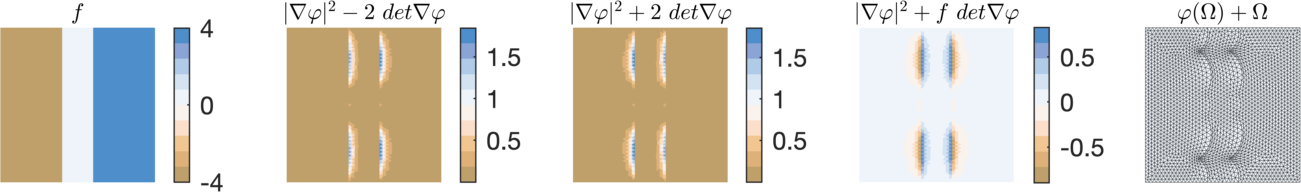}
\caption{$\rho=1/2$:  $\bf{I(\varphi)<0}$
}
\vspace{1em}
\end{subfigure}
\begin{subfigure}[b]{\textwidth}
\centering
\includegraphics[width=\textwidth]{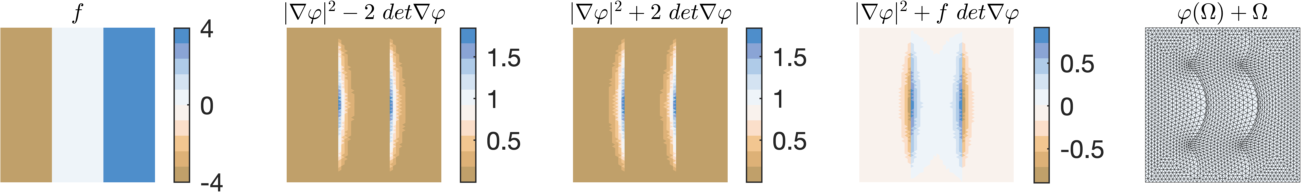}
\caption{$\rho=1$:  $\bf{I(\varphi)<0}$
}
\vspace{1em}
\end{subfigure}
\begin{subfigure}[b]{\textwidth}
\centering
\includegraphics[width=\textwidth]{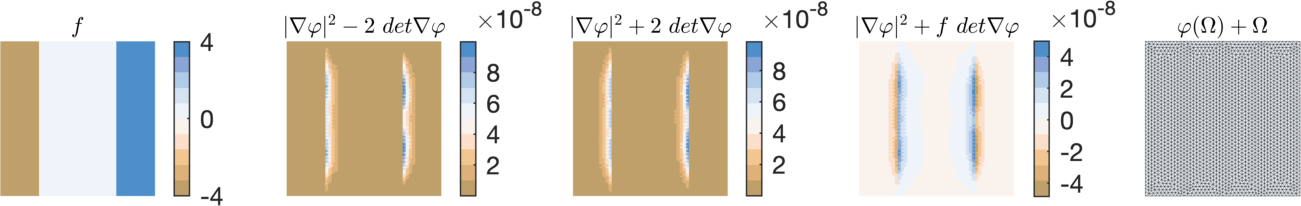}
\caption{$\rho=2$:  $\bf{I(\varphi)>0}$
}
\end{subfigure}
\caption{A sequence of minimizers: an insulation problem with variable width parameter $\rho$ and constant deflection angle $\alpha=0$.  The boundary $\Gamma$, which is defined in \eqref{gammabdry}, lies along the intersection of the gold and light blue regions in each figure in the leftmost column.}
\label{fig:insulation_intro1}
\end{figure}

\begin{figure}
\begin{subfigure}[b]{\textwidth}
\centering
\includegraphics[width=\textwidth]{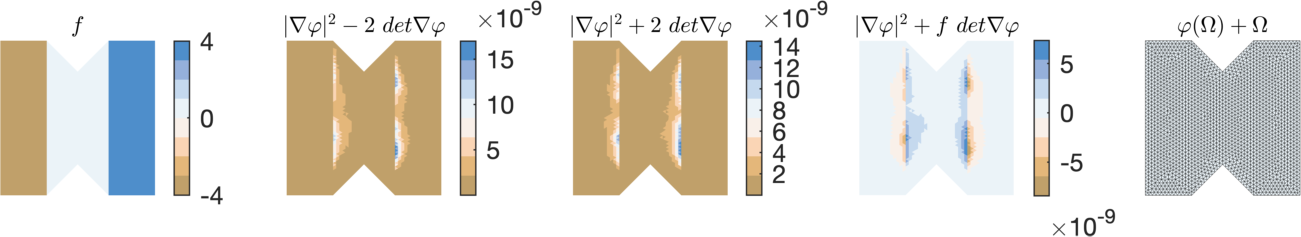}
\vspace{-1cm}
\caption{$\alpha = -\pi/4 : \bf{I(\varphi)>0}$}
\vspace{1em}
\end{subfigure}
\begin{subfigure}[b]{\textwidth}
\centering
\includegraphics[width=\textwidth]{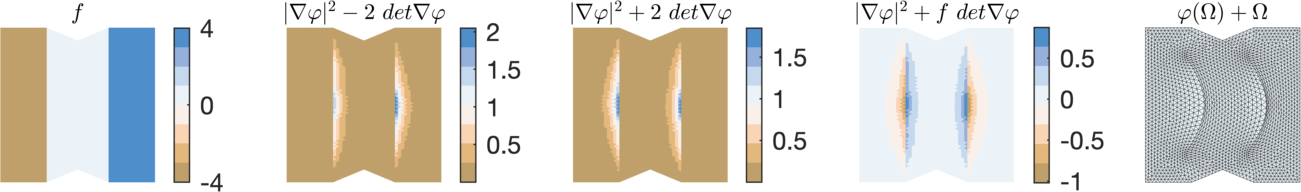}
\caption{$\alpha = -\pi/8 : \bf{I(\varphi)<0}$}
\vspace{1em}
\end{subfigure}
\begin{subfigure}[b]{\textwidth}
\centering
\includegraphics[width=\textwidth]{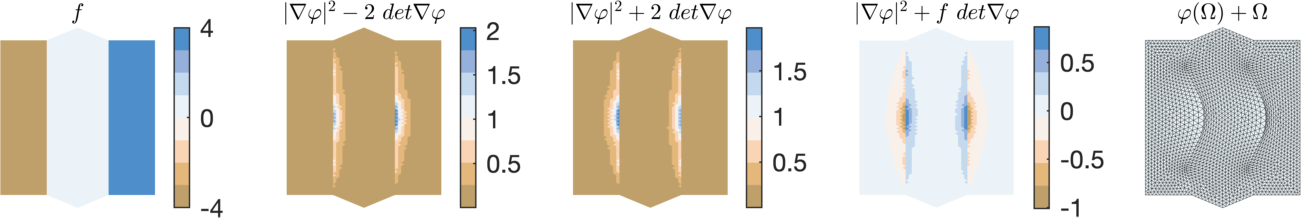}
\caption{$\alpha = \pi/8: \bf{I(\varphi)<0}$}
\vspace{1em}
\end{subfigure}
\begin{subfigure}[b]{\textwidth}
\centering
\includegraphics[width=\textwidth]{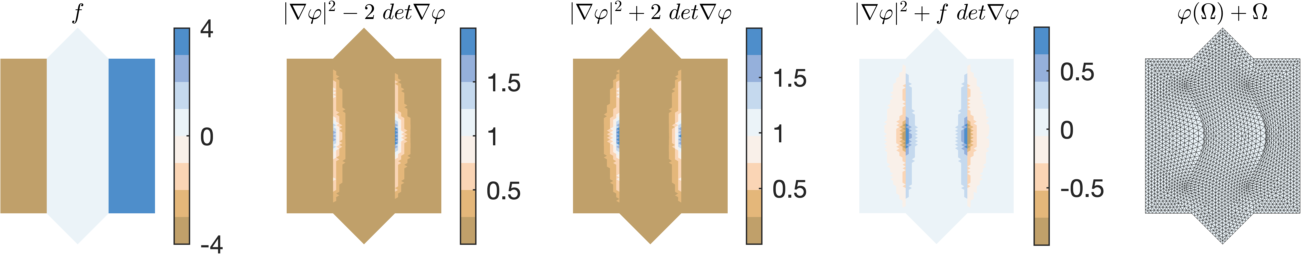}
\caption{$\alpha = \pi/4: \bf{I(\varphi)<0}$}
\end{subfigure}
\caption{A sequence of minimizers: an insulation problem with a fixed width $w_0=0.8, w_4=w_{-4}=0.6$ and a variable deflection angle $\alpha$.}
\label{fig:insulation_intro2}
\end{figure}

One possible heuristic explanation is as follows.  First note that we may assume without loss of generality that $\varphi$ is symmetric about the $x_2-$axis, which bisects the region $\{f=0\}$, and so confine attention to the behavior of $\varphi$ in just the left-hand plane, $L$, say.  Further, since $\varphi$ can be assumed to be harmonic in the regions $\{f=-4\}$ and $\{f=0\}\cap L$, and since $\det \nabla \varphi$ is a null Lagrangian, the value of 
\begin{align}\label{halfaneye}
    I_{L}(\varphi):=\int_{\{f=-4\}}|\nabla \varphi|^2-4 \det \nabla \varphi \dx + \int_{\{f=0\}\cap L}|\nabla \varphi|^2\dx
\end{align}
is completely determined by the behavior of $\varphi$ along the boundary sets 
\begin{align}\label{gammabdry}\Gamma:=\partial (\{f=-4\}) \cap \partial (\{f=0\})
\end{align}and the $x_2$-axis $\partial(L \cap \{f=0\})$.  Recall that $\varphi$ vanishes on $\partial Q$.  The second column of Figures \ref{fig:insulation_intro1} (A), (B), (C) and (D) shows that in the region $\{f=-4\}$ low-energy $\varphi$ tend to behave more conformally than not, which makes sense in view of the integrand of \eqref{halfaneye}.  This, in turn, produces a trace $\varphi\arrowvert_{_{\Gamma}}$ that helps determine\footnote{Leaving aside the observation that, by symmetry and harmonicity, $\partial_1\varphi$ should vanish on $\partial(L \cap \{f=0\})$. } the Dirichlet energy $$D_0(\varphi):=\int_{L\cap\{f=0\}}|\nabla \varphi|^2 \dx.$$  The two quantities $D_0(\varphi)$ and
\begin{align}\label{confe} \int_{\{f=-4\}}|\nabla \varphi|^2-4 \det \nabla \varphi \dx\end{align}
compete.  In Figure \ref{fig:insulation_intro1} (A), the higher frequency trace $\varphi\arrowvert_{\Gamma}$ presumably lowers the energy \eqref{confe} without incurring a large Dirichlet cost $D_0(\varphi)$ in the neighboring region $L\cap \{f=0\}$, and in such a way that the negative term \eqref{confe} dominates.  In Figure \ref{fig:insulation_intro1} (C), the same mechanism leads not only to negative energy but also to a lower frequency trace: for the domain in (C), the Dirichlet energy of the traces seen in (A) and (B), for example, is too large.  Finally, in Figure \ref{fig:insulation_intro1} (D), the width of the $\{f=0\}$ region is such that even the `low frequency' trace $\varphi_{_\Gamma}$ has a larger $D_0(\varphi)$ than the negative contribution of \eqref{confe}, and the functional cannot become negative.

\subsubsection{Variable angle}
In Figure \ref{fig:insulation_intro2}, a similar pattern is observed, but in this case, the angle of deflection $\alpha$, which is defined in Figure \ref{pic:alphadescription}, now plays the role of the changing $\rho$ featured in Figure \ref{fig:insulation_intro1}.  Since the quantities $w_0$ and $w_4$ are fixed, the domain featured in Figure \ref{fig:insulation_intro2} (B) is included in that featured in (C) and (D), so any $\varphi$ causing (HIM) to fail in case (B) will also cause it to fail in cases (C) and (D).

\subsection{Grid of $m \times m$ points}
To conclude our brief numerical investigation, we now consider a sequence of functionals inspired by the point-contact insulation problem discussed in Section \ref{-4040}.  
In the notation introduced earlier, let $$f_2 = \sqrt{8} \chi_{Q_1} - \sqrt{8} \chi_{Q_3}$$
and recall that, by Proposition \ref{foursquares1}, the functional
$$ I_2(\varphi,f_2):=\int_{Q} |\nabla \varphi|^2 +f_2(x)\det \nabla \varphi \dx $$
obeys $I_2(\varphi,f_2)\geq 0$
for all candidate $\varphi$, and, moreover, that $\sqrt{8} \approx 2.8284$ is the largest value for which this inequality holds.

To check how sharply this bound can be approximated numerically, we consider a sequence of minimization problems with a modified functional 
$$ I_2(\varphi, \lambda f_2), $$
where $\lambda$ is a real parameter.  According to the analysis above, we expect that $\lambda=1$ will be a critical value in the sense that $I_2(\varphi, \lambda f_2) \geq 0$ should hold for all test functions $\varphi$ when $|\lambda|\leq 1$, whereas this should fail to be the case when $|\lambda|>1$.  Using a bisection method for $\lambda \in [1,2]$, we searched numerically for $\lambda$ such that $ I_2(\varphi, \lambda f_2) <0 $ occurs for some $\varphi$. Each such $\lambda$ is an upper bound on the `true' critical value of 1.  The results are shown in Figure \ref{lamgrid}.  We see that the computed values of $\lambda$ depend on the choice of the computational mesh and that the finer the mesh, the closer to $1$ were the approximations.  The smallest value $\lambda=1.0859$ was obtained on a regular mesh with 65536 triangular elements. 
\begin{figure}
\includegraphics[width=0.5\textwidth]{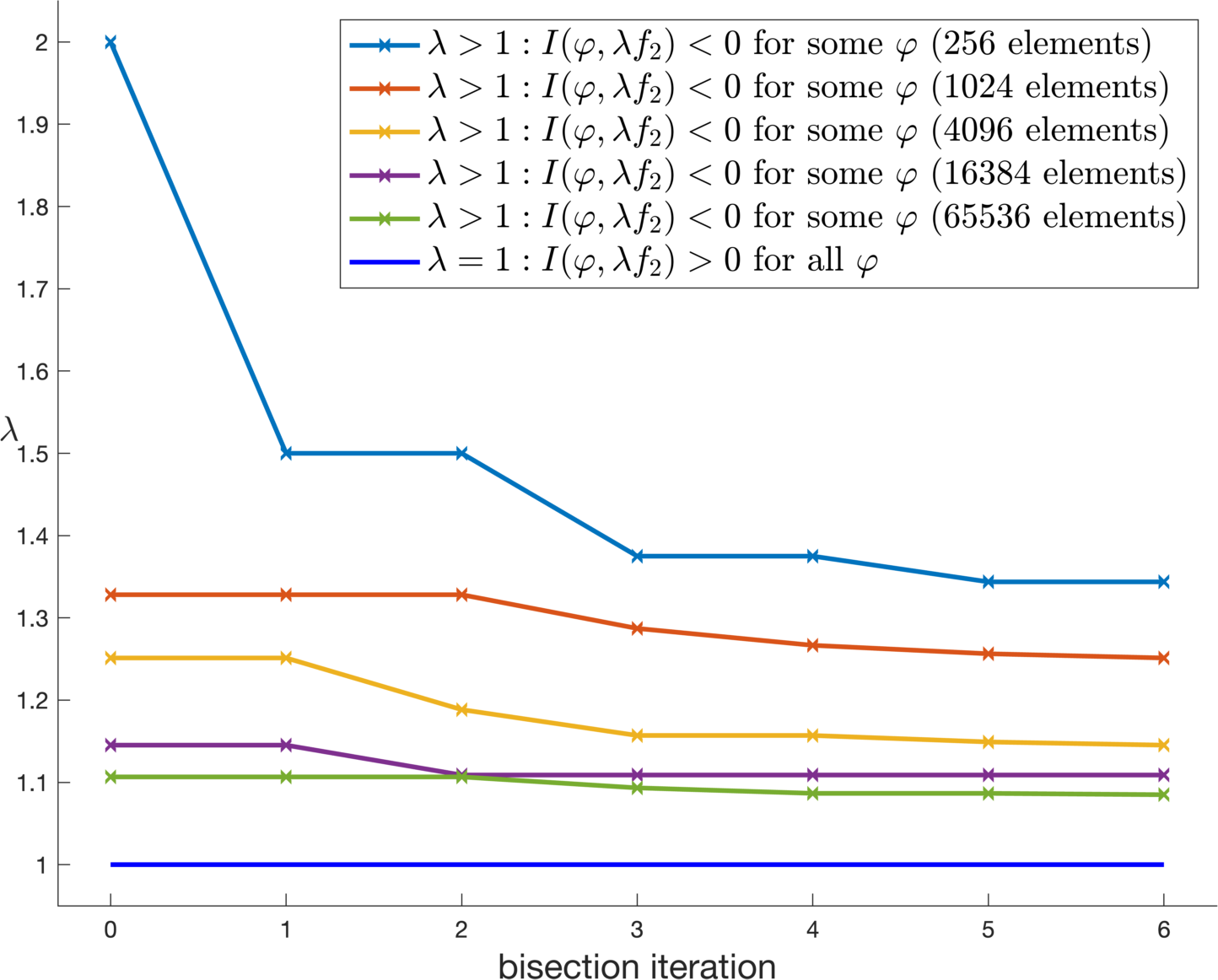}
\caption{Values of $\lambda$ ensuring $I_2(\varphi, \lambda f_2)<0$ for some $\varphi$, which are produced by the bisection method applied to 5 nested triangular meshes.
}\label{lamgrid}
\end{figure}

It is natural to ask whether the (HIM) inequality
$$I_2(\varphi,f_2)\geq 0 \quad \varphi \in W_0^{1,2}(Q;\R^2),$$ which has $\delta f=2 \sqrt{8}$, can be improved by subdividing the domain $Q$ more finely, redefining $f$ suitably, and thereby producing larger values of $\delta f$.   


Accordingly, consider a sequence of piecewise constant functions $f_2, f_3, \ldots$ that are chosen to replicate along the main diagonal of $Q$ the maximum jump of $2 \sqrt{8}$ that features in $f_2$.  

For illustration, the leftmost column of Figure \ref{gridmins} shows, in graphical form, the values taken by $f_5$ and the subsquares of $Q$ on which the various values are taken.  The other $f_m$ are constructed similarly, and they have the properties that
\begin{enumerate}
    \item[(i)] the largest value taken by each $f_m$ is $\sqrt{8}(m-1)$,
   \item[(ii)] the largest jump in $f_m$ is $\delta{f_m}=2 \sqrt{8}(m-1)$.  
\end{enumerate}
The figures in the four other columns show observed features of $\varphi$ for which $I_2(\varphi,f_5)<0$, where, more generally, 
$$I_2(\varphi,f_m):=\int_{Q} |\nabla \varphi|^2 + f_m(x)\det \nabla \varphi \dx, \quad m=2,3,\ldots.$$

\begin{figure}[H]
\includegraphics[width=\textwidth]{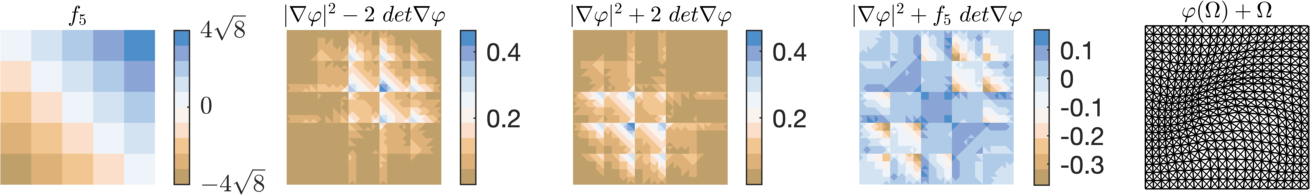}
\caption{Leftmost column: the distribution of values taken by the piecewise constant function $f_5$. Other columns: features of $\varphi$ such that $I_2(\varphi,f_5)<0$.  }
\label{gridmins}
\end{figure}


Numerical experiments with outputs similar to those of Figure \ref{gridmins} show that (HIM) appears to fail in the cases $m=3,4,5$.  We infer that other than for the $2 \times 2$ grid, the maximum `diagonal' jump of $2\sqrt{8}$ along the leading diagonal of $Q$ is too large for (HIM) to hold.  Before modifying the approach slightly, let us note in relation to Figure \ref{gridmins} that when $m=5$,
 \begin{itemize}
\item[(i)] low-energy $\varphi$ appear to concentrate their most significant `non-zero behaviour' near intersection points of sets of four subsquares; 
\item[(ii)] by superimposing neighbouring figures in the second and third columns, we see that the 
regions described in (i) tend to be composed of mutually exclusive `patches' of conformal/anticonformal gradients. 
\end{itemize}
Similar patterns were observed for other choices of $m$.

Let us now replace each $f_m$ by $\lambda f_m$, where $\lambda > 0$, and seek the largest $\lambda$ for which 
\begin{align}\label{iphilamn}I(\varphi,\lambda f_m)=\int_{Q} |\nabla \varphi|^2 + \lambda f_m(x)\det \nabla \varphi \dx\end{align}
obeys $I(\varphi,\lambda f_m) \geq 0$ for all $\varphi$ in $W^{1,2}_0(Q;\R^2)$.  The pointwise Hadamard inequality implies that $\lambda$ such that that $\sqrt{8}\lambda(m-1) = 2$ will be such that (HIM) holds, so the set $$S(m):=\{\lambda>0: \ I(\varphi,\lambda f_m) \geq 0 \ \textrm{for all} \  \varphi \in W^{1,2}_0(Q;\R^2)\}$$
contains $(\sqrt{2}(m-1))^{-1}$. Let $\lambda_{\textrm{crit}}(m)=\sup S(m)$ and note that if $\lambda_{\textrm{crit}}(m)=\infty$ then we may take $\delta (\lambda f_m)$ arbitrarily large such that (HIM) holds.  Otherwise, $\lambda_{\textrm{crit}}(m)$ is finite, and in fact bounded above by $1$ in all cases $m \geq 3$ that we have tested. See Figure \ref{lamcrit} for the cases $m=3,\ldots,25$, for example.  

\begin{lemma}\label{Iorder}
Let $I(\varphi;\lambda f_m)$ be given by \eqref{iphilamn}, let $ \mu > \lambda$ and suppose that $I(\varphi;\lambda f_m)<0$ for some $\varphi \in W^{1,2}_0(Q,\R^2)$.  Then
\begin{align}\label{order}I(\varphi;\mu f_m)<I(\varphi;\lambda f_m).\end{align}
\end{lemma}
\begin{proof}
The assumptions imply that 
$$\int_{Q} \lambda f_m \det \nabla \varphi \dx < - \int_Q |\nabla \varphi|^2\dx ,$$
the right-hand side of which must be negative.  Hence
\begin{align*}
\int_{Q} \mu f_m \det \nabla \varphi \dx < \int_{Q} \lambda f_m \det \nabla \varphi \dx,
\end{align*}
from which \eqref{order} follows easily.
 \end{proof}

We infer that $S(m)$ is the interval $(0,\lambda_{\textrm{crit}}(m)).$  Indeed, if this were not the case then there are points $a<b<c$ such that $a,c\in S(m)$ and $b\notin S(m)$. If $b \notin S(m)$ then there is $\varphi$ such that $I(\varphi,b f_m)<0$, and hence, by Lemma \ref{Iorder}, $I(\varphi; c f_m)<0$, contradicting the assumption that $c \in S(m)$. 
By referring to Figure \ref{lamcrit}, it seems to hold that $S(m)\subset [0,1)$.   

Using the bisection method described earlier, we calculate approximations, labelled $\lambda_{\textrm{approx}}(m)$, to $\lambda_{\textrm{crit}}(m)$ for some of $m=1, \ldots, 25$, the results of which are shown in Figure \ref{lamcrit}. 

\begin{figure}[H]
\includegraphics[width=0.8\textwidth]{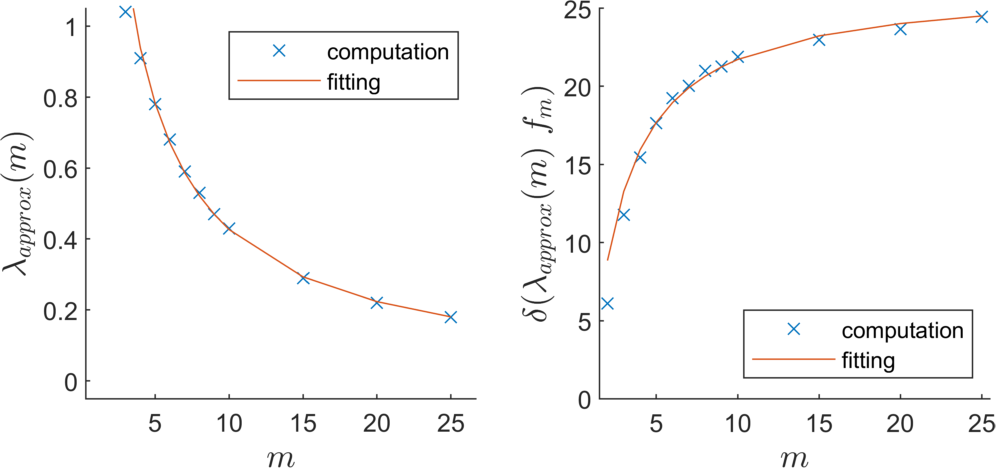}
\caption{Left plot: $\lambda_{\textrm{crit}}(m)$ against $m$. Right plot: $\delta(\lambda_{\textrm{crit}}(m)f_m)$ against $m$. Here, $m^2$ is the number of equally-sized squares into which $Q$ is divided when constructing $f_m$.}\label{lamcrit}
\end{figure}

The curve fitting in the left plot suggests that $\lambda_{\textrm{approx}}(m) \sim \frac{c_0}{m+1}$ with $c_0 \approx 4.6924$.  This further implies, as one can see from the right plot, that $\delta( \lambda_{\textrm{approx}}(m) f_m) \sim 24.5023 - \mathrm{O}(\frac{1}{m+1})$.  Thus, supposing that $
\lambda_{\textrm{approx}}(m) \simeq \lambda_{\textrm{crit}}(m)$, (HIM) appears to be consistent with jumps in $f$ of order $\sim24.5023$.  

\section{Sequential weak lower semicontinuity and nonnegativity of $I_n$}

Here we will study the relationship between the nonnegativity of $I_n$ and its sequential weak lower semicontinuity on $W^{1,n}_0(Q;\R^n)$. Given $f\in L^\infty (Q)$ let us denote for any  $x\in Q$ and  any $F\in\R^{n\times n}$
\begin{align}
h(x,F)=|F|^n+ f(x)\det F\ .
    \end{align}
    It is easy to see that $h:Q\times\R^{n\times n}\to\R$ is a Carath\'{e}odory integrand and   $h(x,\cdot)$ is quasiconvex for almost every $x\in Q$ \cite{morrey}.  Moreover,  $h(x,\cdot)$ is $n$-Lipschitz, i.e.,  there is $\alpha>0$ such that for almost every $x\in Q$ and every $F_1,F_2\in\R^{n\times n}$
    
    \begin{align}\label{n-Lipschitz}|h(x,F_1)-h(x,F_2)|\le \alpha(1+|F_1|^{n-1}+ |F_2|^{n-1})|F_1-F_2|\ ,\end{align}

This follows from  \cite[Prop.~2.32]{dacorogna} and the fact that $f$ is uniformly  bounded.
The following lemma proved in \cite{fmp} is crucial for further results in this section. 

\begin{lemma}\label{fons}
Let   $Q\subset\R^n$ be as bounded domain  and let $( \varphi_k)_{k\in\N}\subset W^{1,n}(Q;\R^n)$ be bounded. Then there is a subsequence $(\varphi_j)_{j\in\N}$ and a sequence $(z_j)_{j\in\N}\subset W^{1,n}(Q;\R^n)$ such that
\begin{align}\label{rk}
\lim_{j\to\infty} \mathcal{L}^n(\{x\in Q:\ z_j(x)\ne \varphi_j(x)\mbox{ or }  \nabla z_j(x)\ne \nabla \varphi_j(x)\})=0
\end{align}
and $(|\nabla z_j|^n)_{j\in\N}$ is relatively weakly compact in $L^1(Q)$.
\end{lemma}

Assume that $\varphi_k\wto\varphi$ in $W_0^{1,n}(Q;\R^n)$. Extracting a non-relabeled subsequence we can write $\varphi_k=z_k+w_k$ where $(z_k)_k$ is defined by Lemma~\ref{fons}, Hence, $z_k\wto\varphi$ in $W^{1,n}(Q;\R^n)$ and $\nabla w_k\to 0$ in measure as $k\to\infty$. 
Assume that for $\ell\in\N$ large enough  $\eta_\ell:Q\to[0,1]$ is a cut-off function such that $\eta_\ell(x)=1$ if dist$(x,\partial Q)\ge \ell^{-1}$, $\eta_\ell(x)=0$ on $\partial Q$ and $|\nabla\eta_l|\le C\ell$ for some $C>0$. We define $u_{k\ell}=\eta_\ell z_k+(1-\eta_\ell)\varphi$. Following \cite[Lemma~8.3]{Pedregal}
we get a sequence $(\tilde z_\ell)_\ell=(u_{k(\ell)\ell})_\ell$ such that 
\begin{align}\label{rk1}
\lim_{j\to\infty} \mathcal{L}^n(\{x\in Q:\ \tilde z_j(x)\ne z_j(x)\mbox{ or }  \nabla \tilde z_j(x)\ne \nabla z_j(x)\})=0
\end{align}
and $(|\nabla \tilde z_j|^n)_{j\in\N}$ is relatively weakly compact in $L^1(Q)$. 
Altogether, it follows that $\lim_{j\to\infty}\|\nabla z_j-\nabla \tilde z_j\|_{L^n(Q;\R^n)}=0$ by the Vitali convergence theorem. 
It is easy to see using \eqref{n-Lipschitz} together with Lemma~\ref{fons} and the previous calculations that 
\begin{align}
    \liminf_{k\to\infty}I_n(\varphi_k)& = \liminf_{k\to\infty}(I_n(z_k)+I_n(w_k))=\liminf_{k\to\infty}(I_n(\tilde z_k)+I_n(w_k))\nonumber\\ &\ge\liminf_{k\to\infty}I_n(\tilde z_k) +\liminf_{k\to\infty}I_n(w_k)\ .
\end{align}
It follows from \cite[Thm~3.1 (i)]{kroemer} that $\liminf_{k\to\infty}I_n(\tilde z_k)\ge I_n(\varphi)$.  If, additionally, we have \begin{align}\label{swlsc-1}\liminf_{k\to\infty}I_n(w_k)\ge I_n(0)=0\end{align} then 
\begin{align}
    \liminf_{k\to\infty}I_n(\varphi_k)\ge I_n(\varphi),
    \end{align}
    i.e., $I_n$ is sequentially weakly lower semicontinuous on $W^{1,n}(Q;\R^n)$.
    Obviously, \eqref{swlsc-1} holds if $I_n\ge 0$. On the other hand, 
    If $I_n$ is sequentially weakly lower semicontinuous on $W^{1,n}_0(Q;\R^n)$ then 
    \eqref{swlsc-1} must hold for every $(w_k)\subset W^{1,n}_0(Q;\R^n)$ converging to zero in measure.
    Altogether, it yields the  following proposition.
    \begin{proposition}\label{Prop:swlsc}
    Let $f\in L^\infty(Q)$ and let $I_n\ge 0$ on $W^{1,n}_0(Q;\R^n)$.  Then $I_n$ is sequentially weakly lower semicontinuous on $W^{1,n}_0(Q;\R^n)$ if and only if 
    \eqref{swlsc-1}  holds for every $(w_k)\subset W^{1,n}_0(Q;\R^n)$ converging to zero in measure. In particular, $I_n$ is sequentially weakly lower semicontinuous on $W^{1,n}_0(Q;\R^n)$ if $I_n\ge 0$ on $W^{1,n}_0(Q;\R^n)$.
    \end{proposition}

    
\section{Acknowledgement}
This work was supported by the Royal Society International Exchanges Grant IEES\ R3\ 193278. MK and JV are thankful to the Department of Mathematics of the University of Surrey for hospitality during their stays there. Likewise, JB is very grateful to \'{U}TIA of the Czech Academy of Sciences for hosting his visits.  All three authors thank Alexej Moskovka and Jonathan Deane for fruitful discussions.

\appendix
\section*{Appendix}
\renewcommand{\thesection}{A} 
It is important for the validity of the arguments leading to Proposition \ref{sum:4+} to establish that the conformal map $\psi$, referred to in \eqref{bc:h2}, exists.  This is the purpose of the following section.
\setcounter{equation}{0}
\subsection{Construction of the conformal map $\psi$}

Let $Q$ be the square in $\mathbb{C}$ with vertices at $0,1,1+i$ and $i$, and recall that $D$ stands for the unit disk in $\mathbb{C}$.   We construct here a conformal map $\phi: D \to Q$ such that the (unique) continuous extension $\tilde{\phi}: \bar{Q} \to \bar{D}$ possesses certain symmetries that are needed in the course of the proof of Proposition \ref{sum:4+}.  From now on let us write $\phi$ for both the map and its extension.  
 Then $\phi$ will be such that for each fixed $\theta$ the four points in the ordered set $\{e^{i\theta},-e^{-i\theta},-e^{i \theta},e^{-i \theta}\}$ are mapped to four points in the ordered set $\{\phi_1,\phi_2,\phi_3,\phi_4\} \subset \partial Q$ with the properties that 
\begin{enumerate} \item[(a)] $\phi_1$ and $\phi_2$ are mirror images of one another in the line $[1,i]$;
\item[(b)] $\phi_3$ and $\phi_4$ are mirror images of one another in the line $[1,i]$;
\item[(c)] $\phi_1$ and $\phi_4$ are mirror images of one another in the line $[0,1+i]$;
\item[(d)] $\phi_2$ and $\phi_3$ are mirror images of one another in the line $[0,1+i]$,
\end{enumerate}
as illustrated below.  
\begin{lemma}\label{map:psi} There exists a conformal map $\phi: D \to Q$ with symmetry properties (a)-(d) and whose unique homeomorphic extension, also denoted $\phi$, obeys $\phi(1)=1+i$, $\phi(i)=i$, $\phi(-1)=0$ and $\phi(-i)=1$, together with 
\begin{align}\label{imagepsi}
\begin{array}{r l}
\{\phi(e^{i\mu}): \ 0 \leq \mu \leq \pi/2\} \hspace{-3mm} & = [i+1,i],  \vspace{1mm}\\ 
\{\phi(e^{i\mu}): \ \pi/2 \leq \mu \leq \pi \}\hspace{-3mm} & = [i,0], \vspace{1mm}\\
\{\phi(e^{i\mu}): \ \pi  \leq \mu \leq 3\pi/2 \} \hspace{-3mm} & = [0,1], \vspace{1mm}\\
\{\phi(e^{i\mu}): \ 3\pi/2 \leq \mu \leq 2\pi \} \hspace{-3mm} & = [1,1+i].\vspace{1mm} 
\end{array}
\end{align}
\end{lemma}
\begin{proof} Let $H$ be the upper half-plane in $\C$, let $H_0$ be the first quadrant $\{z \in \C: \ \mathrm{Im}\,z >0, \ \mathrm{Re}\,z >0\}$, and let $T$ be the interior of the triangle with vertices at $0$, $1$ and $1+i$.   To start with, we follow \cite[Section 5.3]{Palka} and let $g_0: H_0 \to T$ be the unique conformal map whose homeomorphic extension to $\bar{H_0}$, again denoted $g_0$ for brevity, satisfies $g_0(0)=0$, $g_0(1)=1$ and $g_0(\infty)=1+i$.  Note that $g_0$ takes the imaginary axis in $H$, $\alpha$ say, to $[0,1+i]$.  Let $g:H \to Q$ be obtained by reflection in the imaginary axis, and note in particular that, because $g_0(1)=1$, it must be that $g(-1)=i$, which is the mirror image of $g_0(1)$ in $g_0(\alpha)$.   We are now in possession of the preimage under $g$ of each vertex of $Q$, and since the conformal map $G: H \to Q$ that is prescribed by the Schwarz-Christoffel formula (e.g., see \cite[Theorem 5.6]{Palka})
\begin{align}\label{form:SC} G(z)=A \int_0^z \xi^{-\frac{1}{2}} (\xi-1)^{-\frac{1}{2}}  (\xi+1)^{-\frac{1}{2}}  \, d\xi
\end{align}
agrees, for a suitable constant $A$, with $g$ at a triple of oriented points on the boundary of $H$, it follows by \cite[Theorem 4.12]{Palka} that $g=G$ in $H$.   It will shortly be helpful to note that if $[-\infty,a]$, $[a,b]$ and $[b,\infty]$ are understood as subintervals of $\R \cup \{\pm \infty\}$ embedded in $\tilde{\C}$ then
\begin{align} \label{imageg}  \begin{array}{r l} g([-\infty,-1]) \hspace{-3mm}& = [1+i,i], \vspace{1mm} \\ 
g([-1,0]) \hspace{-3mm}& =[i,0], \vspace{1mm} \\ 
g([0,1]) \hspace{-3mm}& = [0,1], \ \mathrm{and}  \vspace{1mm} \\ \vspace{1mm}
g([1,\infty])\hspace{-3mm} & = [1,1+i].\end{array}
\end{align}

Let $C_1=\{z \in \bar{H}: \ |z|=1\}$.  We claim that $g(C_1)=[1,i]$.  To see it, we appeal directly to \eqref{form:SC}, which, after a short calculation, shows that if $z=e^{i\mu}$ with $0 \leq \mu \leq \pi$, 
\begin{align*}
G(e^{i \mu})= (i-1)\frac{K(\mu)}{K(\pi)} + 1
\end{align*} 
where $K(\mu)$ is real, finite and $K(0)=0$.   The claim is immediate. 

Now define $v: D \to H$ by $v(z)=i\left(\frac{1+z}{1-z}\right)$ and note the following easily-verified facts about $v$:
\begin{align}\label{imagev}\begin{array}{r  l} \{v(e^{i\mu}): \ 0 \leq \mu \leq \pi/2\} \hspace{-3mm} & = [-\infty,-1], \vspace{1mm} \\
\{v(e^{i\mu}): \ \pi/2 \leq \mu \leq \pi \} \hspace{-3mm}& = [-1,0],  \vspace{1mm} \\
\{v(e^{i\mu}): \ \pi  \leq \mu \leq 3\pi/2 \}  \hspace{-3mm}& = [0,1], \vspace{1mm} \\
\{v(e^{i\mu}): \ 3\pi/2 \leq \mu \leq 2\pi \} \hspace{-3mm}& = [1,\infty], \vspace{1mm} \\
v([-i,i]) \hspace{-3mm} & = C_1,
\end{array}
\end{align}
where, in the first four lines, the convention of \eqref{imageg} concerning image sets remains in force.

Let $\phi: D \to Q$ be given by  
\begin{align}\label{def:psi}\phi:=g \circ v,\end{align}
and note that by combining \eqref{imageg} with \eqref{imagev} we obtain \eqref{imagepsi}.  See Figure \ref{Byrd2}.
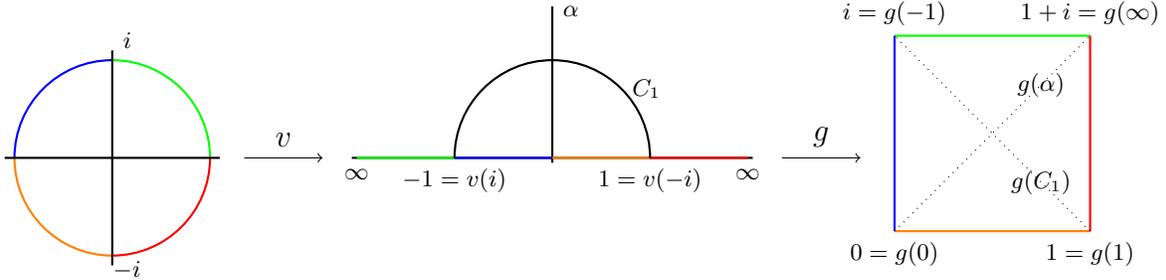
\begin{figure}[ht]
\centering
\begin{minipage}{\textwidth}
\centering
\begin{tikzpicture}[scale=1.3]

\begin{scope}[shift={(-3,0.0)}]
\def\ac45{0.707106}
\node (M) at ( 3.0, 0.0) {}; 
\draw[thick,green] (4,0) arc[start angle=0,end angle=90,radius=1];
\draw[thick,blue] (3,1) arc[start angle=90,end angle=180,radius=1];
\draw[thick,orange]  (2,0) arc[start angle=180,end angle=270,radius=1];
\draw[thick,red]  (3,-1) arc[start angle=270,end angle=360,radius=1];
\node (A) at  (-\ac45+3,-\ac45) {};
\node (A') at  (-\ac45+3,\ac45) {};
\node (B) at (\ac45+3,-\ac45) {};
\node (B') at (\ac45+3,\ac45) {};
\draw[thick] (-1.1+3,0) -- (1.1+3,0);
\draw[thick] (0+3,-1.1) -- (0+3,1.1);
\node[right] (U) at  (3,1.2) {\scriptsize $i$};
\node[right] (U) at  (2.9,-1.15) {\scriptsize $-i$};

\end{scope}

\begin{scope}[shift={(0.25,0)}]
\draw[->] (1.1,0) -- node[above] {$v$} (1.9,0);
\end{scope}

\begin{scope}[shift={(4.5,0)}]
\def\ac45{0.707106}
\node (M) at ( 0.0, 0.0) {}; 
\node (A) at  (-\ac45,-\ac45) {};
\node (A') at  (-\ac45,\ac45) {};
\node (B) at (\ac45,-\ac45) {};
\node (B') at (\ac45,\ac45) {};
\draw[thick] (-2.05,0) -- (2.05,0);
\draw[thick] (0,-0.05) -- (0,1.55);
\draw[thick,green] (-2,0) -- (-1,0);
\draw[thick,blue] (-1,0) -- (0,0);
\draw[thick,orange] (0,0) -- (1,0);
\draw[thick,red] (1,0) -- (2,0);

\node[right] (U) at  (0,1.5) {\scriptsize $\alpha$};
\node[right] (U) at  (\ac45,\ac45) {\scriptsize $C_1$};
\node[above] (U) at  (0,1) {};
\node[below] (U) at  (0,-1) {};
\node[below] (U) at  (-2,0) {\scriptsize $\infty$};
\node[below] (U) at  (-1,0) {\scriptsize $-1=v(i)$};
\node[below] (U) at  (1,0) {\scriptsize $1=v(-i)$};
\node[below] (U) at  (2,0) {\scriptsize $\infty$};
\draw[thick,black]  (1,0) arc[start angle=0,end angle=180,radius=1];
\end{scope}

\begin{scope}[shift={(2.75,0)}]
\draw[->] (4.1,0) -- node[above] {$g$} (4.9,0);
\end{scope}

\begin{scope}[shift={(8,0.25)}]
\draw[thick,green] (2,1) -- (0,1); 
\draw[thick,blue] (0,1) -- (0,-1); 
\draw[thick,orange] (0,-1) -- (2,-1); 
\draw[thick,red] (2,-1) -- (2,1); 
\node[below] (U) at  (0,-1) {\scriptsize $0=g(0)$};
\node[below] (U) at  (2,-1) {\scriptsize $1=g(1)$};
\node[above] (U) at  (0,1) {\scriptsize $i=g(-1)$};
\node[above] (U) at  (2,1) {\scriptsize $1+i=g(\infty)$};

\draw[dotted] (0,-1) -- (2,1);
\draw[dotted] (2,-1) -- (0,1);
\node[] (U) at  (1.5,0.5) {\scriptsize $g(\alpha)$};
\node[] (U) at  (1.5,-0.5) {\scriptsize $g(C_1)$};

\end{scope}
\end{tikzpicture}
\end{minipage}
\caption{This colour-coded figure records the effect of the conformal map $\phi:=g \circ v$ on the boundary of the unit circle in $\C$.  It can be checked that $[-i,i] \overset{v}{\to} C_1  \overset{g}{\to} [1,i]$, which is one of the diagonals of $Q$, and that $[-1,1)  \overset{v}{\to} \alpha \overset{g}{\to} [0,1+i]$, the other diagonal.  These observations become important when establishing the various symmetries of $\phi$ in Proposition \ref{confexist}.}\label{Byrd2}
\end{figure}
   It only remains to prove the symmetry properties (a)-(d) in the statement of the lemma, and for this it helps to recall the convention set out earlier that if $0 \leq \theta \leq \pi/2$ then $\phi_1=\phi(e^{i \theta})$, $\phi_2=\phi(-e^{-i\theta})$, $\phi_3=\phi(-e^{i \theta})$ and $\phi_4=\phi(e^{-i \theta})$.  By direct calculation, we find that
\begin{align}\label{imagev14}\begin{array}{r l} 
v_1 \hspace{-3mm} &:=v(e^{i\theta})  = \frac{-\sin \theta}{1-\cos \theta},  \vspace{1mm} \\
v_2\hspace{-3mm}&:=v(-e^{-i\theta}) = \frac{-\sin \theta}{1+\cos \theta},  \vspace{1mm} \\
v_3\hspace{-3mm}&:=v(-e^{i\theta}) = \frac{\sin \theta}{1+\cos \theta},  \vspace{1mm} \\
v_4\hspace{-3mm}&:=v(e^{-i\theta}) = \frac{\sin \theta}{1-\cos \theta}.  \vspace{1mm} \\
\end{array}
\end{align}
Since for $\theta >0$ all $v_j$ are real, and $v_1=-v_4$ and $v_2=-v_3$, it must be that $\phi_1=g(v_1)$ and $\phi_4=g(v_4)$ are mirror images of one another in $g(\alpha)=[0,1+i]$.  The same is true for $\phi_2=g(v_2)$ and $\phi_3=g(v_3)$.   The case $\theta=0$ corresponds to $v_1=\infty$ and $v_2=\infty$ in $\tilde{\C}$, and by construction we have $g(\infty)=1+i$.  Hence in all cases symmetries (c) and (d) hold.

To prove symmetries (a) and (b), note from \eqref{imagev14} that, for $\theta >0$, $v_1 v_3 =-1$ and $v_2 v_4 = -1$.  But $v_3=-v_2$, so $v_1 v_2=1$ and $v_4 v_3=1$, from which it is apparent that $v_1$ and $v_2$ are symmetric\footnote{See e.g. \cite[IX Section 2.6] {Palka}} with respect to the part circle $C_1$, and that the same is true of $v_3$ and $v_4$.   Consider the restriction $g_{\arrowvert_{\scriptscriptstyle{D \cap H}}}$ of $g$ to the interior of $C_1$ in $H$, which is the same as the set $D \cap H$.   We can extend $g_{\arrowvert_{\scriptscriptstyle{D \cap H}}}$ using the Schwarz reflection principle (where the reflection is in the part circle $C_1$) into $H\setminus D$, yielding a function $g^*$, say, that is holomorphic in $H$ and which agrees with $g$ on the open set $D \cap H$.  It must therefore be that $g^*=g$ in $H$.   Since $v_1$ and $v_2$ are symmetric with respect to $C_1$ it follows that $g^*(v_1)$ and $g^*(v_2)$ are symmetric\footnote{Strictly, one ought to say symmetric with respect to a circline in $\C$ rather than part of a line.  But by producing $[1,i]$ at both ends, one obtains such a circline.} with respect to $g^*(C_1)=g(C_1)=[1,i]$, using the claim proved after \eqref{imageg}.  Being symmetric with respect to $[1,i]$
in this setting means precisely that $\phi_1=g(v_1)$ and $\phi_2=g(v_2)$ are mirror images of one another in $[1,i]$.  If $\theta=0$ then $v_1=\infty$ and $v_2=0$, and then $\phi_1=g(v_1)=g(\infty)=1+i$ and $g(v_2)=g(0)=0$ are again mirror images in $[1,i]$.  Thus part (a) in the statement of the lemma, and a similar argument leads to (b).  \end{proof}

Recall the definition of the boundary condition $F'_2$ given in \eqref{bc:F2primereal}.  It should be clear that the support of $F_2'$ is contained in $[0,i] \cup [1,1+i]$.  It is also apparent from Figure \ref{Byrd2} that $\phi$ is such that 
\begin{align*}{\phi}^{-1}([0,i]) & \subset \partial D \cap \{z \in \C: \ \mathrm{Re} z < 0, \ \mathrm{Im} z >0\} \\
{\phi}^{-1}([1,1+i]) & \subset \partial D \cap \{z \in \C: \ \mathrm{Re} >  0, \ \mathrm{Im} z < 0\}.
\end{align*}
In particular, the boundary condition $h'(\zeta)=F'_2(\phi_0(\zeta))$, $\zeta \in \partial D$, does not satisfy the condition that if $h(\theta):=h'(e^{i\theta})$ for $0 \leq \theta \leq \pi$ then $\mathrm{spt}\,(h) \subset [\pi/4,3\pi/4]$. But this is easily remedied by replacing $\phi(\zeta)$ with $\phi(e^{i\pi/4} \zeta)$ in the above.   

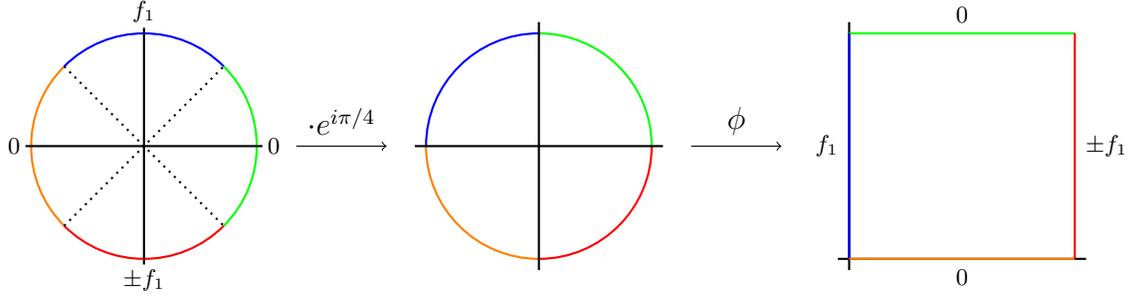
\begin{figure}[ht]
\centering
\begin{minipage}{\textwidth}
\centering
\begin{tikzpicture}[scale=1.5]

\begin{scope}[shift={(0,0)}]
\def\ac45{0.707106}
\node (M) at ( 0.0, 0.0) {}; 
\draw[thick,green] (\ac45,-\ac45) arc[start angle=-45,end angle=45,radius=1];
\draw[thick,blue] (\ac45,\ac45) arc[start angle=45,end angle=135,radius=1];
\draw[thick,orange]  (-\ac45,\ac45) arc[start angle=135,end angle=225,radius=1];
\draw[thick,red]  (-\ac45,-\ac45) arc[start angle=225,end angle=315,radius=1];
\node (A) at  (-\ac45,-\ac45) {};
\node (A') at  (-\ac45,\ac45) {};
\node (B) at (\ac45,-\ac45) {};
\node (B') at (\ac45,\ac45) {};
\draw[thick,dotted] (A.center) -- (B'.center);
\draw[thick,dotted] (A'.center) -- (B.center);
\draw[thick] (-1.05,0) -- (1.05,0);
\draw[thick] (0,-1.05) -- (0,1.05);
\node[above] (U) at  (0,1) {\scriptsize $f_1$};
\node[below] (U) at  (0,-1) {\scriptsize $\pm f_1$};
\node[left] (U) at  (-1,0) {\scriptsize $0$};
\node[right] (U) at  (1,0) {\scriptsize $0$};
\end{scope}

\begin{scope}[shift={(0.25,0)}]
\draw[->] (1.1,0) -- node[above] {$\cdot e^{i\pi/4}$} (1.9,0);
\end{scope}

\begin{scope}[shift={(0.5,0)}]
\def\ac45{0.707106}
\node (M) at ( 3.0, 0.0) {}; 
\draw[thick,green] (4,0) arc[start angle=0,end angle=90,radius=1];
\draw[thick,blue] (3,1) arc[start angle=90,end angle=180,radius=1];
\draw[thick,orange]  (2,0) arc[start angle=180,end angle=270,radius=1];
\draw[thick,red]  (3,-1) arc[start angle=270,end angle=360,radius=1];
\node (A) at  (-\ac45+3,-\ac45) {};
\node (A') at  (-\ac45+3,\ac45) {};
\node (B) at (\ac45+3,-\ac45) {};
\node (B') at (\ac45+3,\ac45) {};
\draw[thick] (-1.1+3,0) -- (1.1+3,0);
\draw[thick] (0+3,-1.1) -- (0+3,1.1);
\end{scope}

\begin{scope}[shift={(0.75,0)}]
\draw[->] (4.1,0) -- node[above] {$\phi$} (4.9,0);
\end{scope}

\begin{scope}[shift={(0.75,0)}]
\draw[thick,black] (5.4,-1) -- (7.6,-1); 
\draw[thick,black] (5.5,-1.1) -- (5.5,1.1); 
\draw[thick,green] (7.5,1) -- (5.5,1); 
\draw[thick,blue] (5.5,1) -- (5.5,-1); 
\draw[thick,orange] (5.5,-1) -- (7.5,-1); 
\draw[thick,red] (7.5,-1) -- (7.5,1); 
\node[left] (U) at  (5.5,0) {\scriptsize $f_1$};
\node[right] (U) at  (7.5,0) {\scriptsize $\pm f_1$};
\node[above] (U) at  (6.5,1) {\scriptsize $0$};
\node[below] (U) at  (6.5,-1) {\scriptsize $0$};
\end{scope}
\end{tikzpicture}
\end{minipage}
\caption{This is a cartoon-like sketch of the boundary conditions imposed on $\partial D$ by `pulling back', via $\psi$, the boundary conditions indicated on $Q$ in the rightmost part of the figure. }\label{Byrd3}
\end{figure}

\begin{definition}\label{finalpsi}(The map $\psi$.) Let $\phi: D \to Q$ be the map constructed in Lemma \ref{map:psi}.  Then define $\psi: D \to Q$ by $$\psi(\zeta) = \phi(e^{i\pi/4} \zeta), \quad \quad \quad \zeta \in D.$$
\end{definition}

To illustrate why symmetries (a)-(d) are needed, consider the maps $F_2': \partial Q \to \C$ and $h': \partial D \to \C$ defined in \eqref{bc:F2primereal} and \eqref{bc:h2} respectively.   We wish $h'$ to have the property that $h'(\zeta)=-h'(-\zeta)$ for all $\zeta \in \partial D$.   For the sake of argument, let $\zeta \in \partial D$ lie in the first quadrant so that  $\eta \in \partial D$ defined by $\eta:=e^{i\pi/4}\zeta$ lies in the fourth quadrant.   Then we may find $\theta \in [0,\pi/2]$ such that $\eta=e^{-i \theta}$. By Lemma \ref{map:psi}, $\phi_4=\phi(e^{-i\theta})$ then belongs to $[1,1+i] \subset \partial Q$.   Let $$\phi_4=1+x_2 i $$ where, without loss of generality, $x_2 \geq 1/2$.  Write $\phi_2=\phi(-e^{-i \theta}),$ which belongs to the line $[0,i]$, as $$\phi_2=(1-y_2) i,$$ and define $$\phi'_2=y_2i$$
to be the mirror image of $\phi_2$ in the line $[i/1, 1+i/2]$. 
We claim that 
\begin{enumerate}\item[(i)] $F'_2(\phi_4)=-F'_2(\phi'_2)$, and hence that 
\item[(ii)] $h'(\zeta)=-h'(-\zeta)$.
\end{enumerate}

To see (i), first observe from \eqref{bc:h2} that $h'(\zeta)=F'_2(\psi(\zeta))=F'_2(\phi(e^{i\pi/4}\zeta))=F'_2(\phi(\eta))=F'_2(\phi_4)$.
Then note that by symmetries (c) and (a) above it must be that $\mathrm{dist} (\phi_4,1+i)=\mathrm{dist}(\phi_2,0)$. See Fig \ref{Byrd4}.  
\begin{figure}[ht]
\centering
\begin{minipage}{1\textwidth}
\centering
\begin{tikzpicture}[scale=1.4]
\node (M) at ( 0.0, 0.0) {}; 
\draw (M.center) circle (1.0cm);

\node[left] (A) at  (-0.86,-1/2) {\scriptsize $(3)  -\!e^{i \theta} $\,};
\node[left] (A') at  (-0.86,1/2) {\scriptsize $(2)  -\! e^{-i \theta}$};
\node[right] (B) at  (0.86,-1/2) {\scriptsize $\, e^{-i \theta} \, (4)$};
\node[right] (B') at  (0.86,1/2) {\scriptsize $\, e^{i \theta} \, (1)$};
\draw[thick, dotted] (A) -- (B');
\draw[thick, dotted] (A') -- (B);
\draw[thick] (-1.2,0) -- (1.2,0);
\draw[thick] (0,-1.2) -- (0,1.2);

\draw[->] (2.8,0) -- node[above] {$\phi$} (3.8,0);

\draw (5,-1) rectangle (7,1); 
\node[below] (C) at (5,-1) {\scriptsize$0$};
\node[left] (C') at (5,1) {\scriptsize$i$};
\node[right] (D') at (7,1) {\scriptsize$1+i$};
\node[below] (D) at (7,-1) {\scriptsize$1$};
\draw[thick, dotted] (5.5,-1) node[below] {\scriptsize$\phi_3$} -- (7,0.5) node[right] {\scriptsize$\phi_4$} -- (6.5,1) node[above] {\scriptsize$\phi_1$} -- (5,-0.5) node[left] {\scriptsize$\phi_2$} -- cycle;
\node[left] at (5,0.5) {\scriptsize$\phi_2'$};
\node at (5,0.5)[circle,fill,inner sep=1pt]{};
\draw[thick, dotted] (5,-1) -- (7,1);
\draw[thick, dotted] (7,-1) -- (5,1);

\end{tikzpicture}
\end{minipage}
\caption{This illustrates (some of) the symmetries of the map $\phi$ that are needed to establish facts about the boundary conditions $h'$ and $h''$ as applied on $\partial D$ in problems of $\ominus-$ and $\oplus-$type.  The convention introduced earlier remains in force, namely that $\phi_1=\phi(e^{i \theta})$, $\phi_2=\phi(-e^{-i\theta})$, $\phi_3=\phi(-e^{i \theta})$ and $\phi_4=\phi(e^{-i \theta})$.  See the discussion that precedes Proposition \ref{confexist} for the details. }\label{Byrd4}
\end{figure}
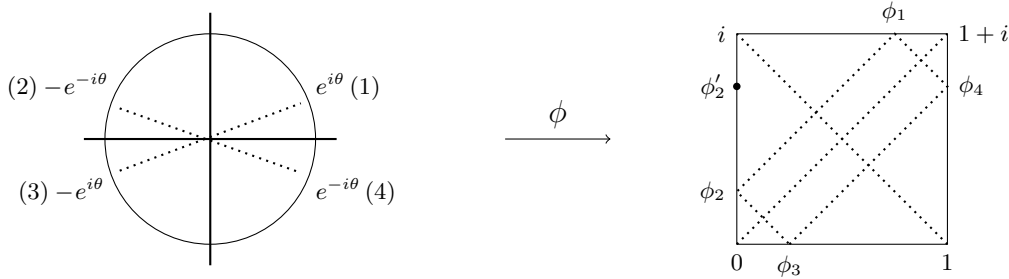

Hence $x_2=y_2$, and so by applying the definition of $F_2'$ given in \eqref{bc:F2primereal} we see that
$$F_2'(\phi_4) = F_2'(1+x_2i)=-f_1(x_2)$$
and
$$F_2'(\phi_2')=F_2'(y_2i) = f_1(y_2)=f_1(x_2).$$ 
This is (i).  To see (ii), note first that $h'(-\zeta)=F_2'(\psi(-\zeta))=F'_2(\phi(-\eta))=F_2'(\phi_2)$ and recall that $h'(\zeta)=F_2'(\phi_4)$.  Using the property that $F_2'(\phi_2)=F_2'(\phi'_2)$ in the first equation and then applying (i), we have $h'(-\zeta)=-h'(\zeta)$, which is (ii).   Thus we are led to the following result.

\begin{proposition}\label{confexist}Let $\psi: D \to Q$ be the conformal map in Definition \ref{finalpsi} and let the maps $h'$ and $h''$ be given by 
\begin{align*} h'(\zeta) & = F'_2(\psi(\zeta)),\\  
h''(\zeta) & = F'_1(\psi(\zeta)), \end{align*}
where $\zeta \in \partial D$, $F'_2$ is given by \eqref{bc:F2primereal} and
\begin{align} \label{bc:F1primereal}
F_1'(z)=\left\{\begin{array}{l l l} f_1(x_2+\frac{1}{2}) & \mathrm{if } & z=x_2 i \in [0,i] \\ 
0 &  \mathrm{if } & z \in [0,1] \cup [i,1+i] \\
 f_1(x_2+\frac{1}{2}) & \mathrm{if } & z=1+x_2i \in [1,1+i]. 
  \end{array}\right. \end{align}
Then $h'(\zeta)=-h'(-\zeta)$ and $h''(\zeta)=h''(-\zeta)$ for all $\zeta \in \partial D$.  Moreover, 
each component of the maps $h'(\theta):=h'(e^{i \theta})$ and $h''(\theta):=h''(e^{i \theta})$, where $0 \leq \theta \leq \pi$, obeys conditions (H1) and (H2). 
\end{proposition}

\end{document}